\documentclass[12pt]{amsart}
\usepackage{amsmath,amsthm,amsfonts,amssymb,mathrsfs}
\usepackage[all]{xy}
\usepackage{eufrak}
\usepackage{amssymb}
\usepackage{latexsym}
\usepackage{amsmath,amsthm}
\usepackage{amssymb}
\usepackage{mathrsfs}
\usepackage{hyperref}
\usepackage{etoolbox}

\makeatletter
\pretocmd{\chapter}{\addtocontents{toc}{\protect\addvspace{15\p@}}}{}{}
\pretocmd{\section}{\addtocontents{toc}{\protect\addvspace{0\p@}}}{}{}
\makeatother

\makeatletter
\def\@tocline#1#2#3#4#5#6#7{\relax
  \ifnum #1>\c@tocdepth 
  \else
    \par \addpenalty\@secpenalty\addvspace{#2}%
    \begingroup \hyphenpenalty\@M
    \@ifempty{#4}{%
      \@tempdima\csname r@tocindent\number#1\endcsname\relax
    }{%
      \@tempdima#4\relax
    }%
    \parindent\z@ \leftskip#3\relax \advance\leftskip\@tempdima\relax
    \rightskip\@pnumwidth plus4em \parfillskip-\@pnumwidth
    #5\leavevmode\hskip-\@tempdima
      \ifcase #1
       \or\or \hskip .5em \or \hskip 1em \else \hskip 1.5em \fi%
      #6\nobreak\relax
    \dotfill\hbox to\@pnumwidth{\@tocpagenum{#7}}\par
    \nobreak
    \endgroup
  \fi}
\makeatother

\newcommand{\C}{\mathbb{C}}
\newcommand{\Z}{\mathbb{Z}}

\newcommand{\Q}{\mathbb{Q}}
\newcommand{\N}{\mathbb{N}}
\newcommand{\F}{\mathbb{F}}

\newcommand{\X}{\mathbb{X}}

\newcommand{\Gal}{\operatorname{Gal}}

\newcommand{\rk}{\operatorname{rk}}

\renewcommand{\sc}{\operatorname{sc}}
\renewcommand{\ss}{\operatorname{ss}}
\newcommand{\der}{\operatorname{der}}

\newcommand{\End}{\operatorname{End}}

\newcommand{\GL}{\mathrm{GL}}

\newcommand{\PSL}{\mathrm{PSL}}

\newtheorem{theorem}{Theorem}[subsection]
\newtheorem{lemma}[theorem]{Lemma}
\newtheorem{proposition}[theorem]{Proposition}

\theoremstyle{definition}
\newtheorem{definition}{Definition}

\newtheorem{remark}[theorem]{Remark}

\newtheorem{innercustomthm}{Theorem}
\newenvironment{customthm}[1]
  {\renewcommand\theinnercustomthm{#1}\innercustomthm}
  {\endinnercustomthm}

\newtheorem{innercustomcor}{Corollary}
\newenvironment{customcor}[1]
  {\renewcommand\theinnercustomcor{#1}\innercustomcor}
  {\endinnercustomcor}

\newtheorem{innercustomrem}{Remark}
\newenvironment{customrem}[1]
  {\renewcommand\theinnercustomrem{#1}\innercustomrem}
  {\endinnercustomrem}

\begin{document}

\title{$\ell$-independence for compatible systems of (mod $\ell$) representations}

\author{Chun Yin Hui}

\email{chunyin.hui@uni.lu, pslnfq@gmail.com}

\address{University of Luxembourg,
Mathematics Research Unit,
6 rue Richard Coudenhove-Kalergi,
L-1359 Luxembourg}

\thanks{Major revisions of the paper were done when I was a postdoctoral fellow at The Hebrew University of Jerusalem supported by Aner Shalev's ERC Advanced Grant no. 247034.}

\begin{abstract} Let $K$ be a number field. For any system
of semisimple mod $\ell$ Galois representations
 $\{\phi_\ell:\mathrm{Gal}(\bar{\Q}/K)\rightarrow\mathrm{GL}_N(\mathbb{F}_\ell)\}_{\ell}$
arising from \'etale cohomology (Definition \ref{arise}), there exists a finite normal extension $L$ of $K$ such that 
if we denote $\phi_\ell(\mathrm{Gal}(\bar{\Q}/K))$ and 
$\phi_\ell(\mathrm{Gal}(\bar{\Q}/L))$ by respectively $\bar\Gamma_\ell$ and $\bar{\gamma}_\ell$ for all $\ell$, and
let $\bar{\mathbf{S}}_\ell$ be
the $\mathbb{F}_\ell$-semisimple subgroup of $\GL_{N,\F_\ell}$ associated to $\bar{\gamma}_\ell$ (or $\bar\Gamma_\ell$) by Nori \cite{Nori} for all sufficiently large $\ell$,
then the following statements hold for all sufficiently large $\ell$: 
\begin{enumerate}
\item[A(i)] The formal character of $\bar{\mathbf{S}}_\ell\hookrightarrow \mathrm{GL}_{N,\mathbb{F}_\ell}$ (Definition \ref{fc}) is independent of $\ell$  and is equal to the formal character of $(\mathbf{G}_\ell^\circ)^{\mathrm{der}}\hookrightarrow \mathrm{GL}_{N,\Q_\ell}$, where $(\mathbf{G}_\ell^\circ)^{\mathrm{der}}$ is the derived group of the identity component of $\mathbf{G}_\ell$, the monodromy group of the corresponding semi-simplified $\ell$-adic Galois representation $\Phi_\ell^{\ss}$.
\item[A(ii)] The non-cyclic composition factors of $\bar{\gamma}_\ell$ and $\bar{\mathbf{S}}_\ell(\mathbb{F}_\ell)$ are identical. 
Therefore, the composition factors of $\bar\gamma_\ell$ are finite simple groups of Lie type of characteristic $\ell$ and cyclic groups.
\item[B(i)] The total $\ell$-rank $\rk_\ell\bar\Gamma_\ell$ of $\bar\Gamma_\ell$ (Definition \ref{rank}) is equal to the rank of $\bar{\mathbf{S}}_\ell$ and is therefore independent of $\ell$.
\item[B(ii)] The $A_n$-type $\ell$-rank $\rk_\ell^{A_n}\bar\Gamma_\ell$ of $\bar\Gamma_\ell$ (Definition \ref{rank}) for $n\in\N\backslash\{1,2,3,4,5,7,8\}$ and the parity of $(\rk_\ell^{A_4}\bar\Gamma_\ell)/4$ are independent of $\ell$. 
\end{enumerate}
\end{abstract}

\maketitle

{\tableofcontents}

\section{Introduction}

Let $K$ be a number field, $\mathscr{P}\subset\N$ the set of prime numbers, and $X$ a complete non-singular variety defined over $K$. For any integer $i$ belonging to $[0,2\mathrm{dim}X]$, the absolute Galois group $\mathrm{Gal}_K:=\mathrm{Gal}(\bar{\Q}/K)$ acts on the $i$th $\ell$-adic \'etale cohomology group $H^i_{\mathrm{\acute{e}t}}(X_{\bar{K}},\Q_\ell)$ for each prime number $\ell\in\mathscr{P}$. The dimension of $H^i_{\mathrm{\acute{e}t}}(X_{\bar{K}},\Q_\ell)$ as a $\Q_\ell$-vector space is independent of $\ell$ and we denote it by $N$. We therefore obtain a system of continuous, $\ell$-adic Galois representations indexed by $\mathscr{P}$:
\begin{equation*}
\{\Phi_\ell:\mathrm{Gal}_K\rightarrow \mathrm{GL}_N(\Q_\ell)\}_{\ell\in\mathscr{P}}
\end{equation*}
which satisfies strict compatibility  (Deligne \cite{Del}) in the sense of Serre \cite[Chapter 1]{Sbk}. There is a conjectural $\ell$-independence \cite{S94} on the images of $\{\Phi_\ell\}$ which has been studied by many people. When $X$ is an elliptic curve without complex multiplication, Serre has proved that the Galois action on the $\ell$-adic Tate module $T_\ell(X)$ is the whole $\mathrm{GL}(T_\ell(X))$ when $\ell$ is sufficiently large by showing that the Galois action $\phi_\ell$ on $\ell$-torsion points $X[\ell]\cong T_\ell(X)/\ell T_\ell(X)$:
\begin{equation*}
\phi_\ell: \mathrm{Gal}_K\rightarrow \mathrm{GL}(X[\ell])\cong\mathrm{GL}_2(\mathbb{F}_\ell)
\end{equation*}
is surjective for $\ell\gg1$ \cite{S72}. This paper is motivated by the idea that the largeness of the $\ell$-adic Galois image $\Gamma_\ell:=\Phi_\ell(\mathrm{Gal}_K)$ can be studied via \emph{taking mod $\ell$ reduction}.  More precisely, given any continuous, $\ell$-adic representation $\Phi_\ell:\mathrm{Gal}_K\rightarrow \GL_N(\Q_\ell)$, one can find a Galois stable $\mathbb{Z}_\ell$-lattice of $\Q_\ell^N$ so that up to some change of coordinates, we may assume $\Phi_\ell(\mathrm{Gal}_K)\subset\mathrm{GL}_N(\mathbb{Z}_\ell)$ since $\mathrm{Gal}_K$ is compact. Then by taking mod $\ell$ reduction $\mathrm{GL}_N(\mathbb Z_\ell)\rightarrow \mathrm{GL}_N(\mathbb F_\ell)$ and semi-simplification, we obtain a continuous, semisimple, mod $\ell$ Galois representation
\begin{equation*}
\phi_\ell: \mathrm{Gal}_K\rightarrow \mathrm{GL}_N(\mathbb{F}_\ell)
\end{equation*}
 which is independent of the choice of the $\mathbb{Z}_\ell$-lattice by Brauer-Nesbitt \cite[Theorem 30.16]{CR}. Denote the mod $\ell$ Galois image $\phi_\ell(\mathrm{Gal}_K)$ by $\bar{\Gamma}_\ell$.

\begin{definition}\label{arise} A system of mod $\ell$ Galois representations 
\begin{equation*}
\{\phi_\ell: \mathrm{Gal}_K\rightarrow \mathrm{GL}_N(\mathbb{F}_\ell)\}_{\ell\in\mathscr{P}}
\end{equation*}
 is said to be \emph{arising from \'etale cohomology} if it is the semi-simplification of a
 mod $\ell$ reduction of the $\ell$-adic system or its dual system:
\begin{equation*}
\{\Phi_\ell: \mathrm{Gal}_K\rightarrow \mathrm{GL}(H^i_{\mathrm{\acute{e}t}}(X_{\bar{K}},\Q_\ell))\}_{\ell\in\mathscr{P}},
\end{equation*}
\begin{equation*}
\{\Phi_\ell: \mathrm{Gal}_K\rightarrow \mathrm{GL}(H^i_{\mathrm{\acute{e}t}}(X_{\bar{K}},\Q_\ell)^\vee)\}_{\ell\in\mathscr{P}}
\end{equation*}
for a complete non-singular variety $X$ defined over $K$ and some $i$, where $H^i_{\mathrm{\acute{e}t}}(X_{\bar{K}},\Q_\ell)^\vee:=\mathrm{Hom}_{\Q_\ell}(H^i_{\mathrm{\acute{e}t}}(X_{\bar{K}},\Q_\ell),\Q_\ell)$.\\
\end{definition}

Let $\rho^{\ss}$ denote the semi-simplification for any finite dimensional representation $\rho$ over a perfect field (well defined by Brauer-Nesbitt \cite[Theorem 30.16]{CR}). Let $\{\Phi_\ell\}$ be a compatible system of $\ell$-adic representations of $\mathrm{Gal}_K$ in Definition \ref{arise}, the algebraic monodromy group  at $\ell$ of the semi-simplified system $\{\Phi_\ell^{\ss}\}$, denoted by $\mathbf{G}_\ell$, is the Zariski closure of $\Phi_\ell^{\ss}(\mathrm{Gal}_K)$ in $\mathrm{GL}_{N,\Q_\ell}$. Then $\mathbf{G}_\ell$ is reductive. Denote the set of non-Archimedean valuations of $K$ and $\bar{K}$ by respectively $\Sigma_K$ and $\Sigma_{\bar{K}}$. The strict compatibility of $\{\Phi_\ell\}$ implies $\{\phi_\ell\}$ is strictly compatible in the following sense.

\begin{definition}\label{comsys} A system of mod $\ell$ Galois representations 
\begin{equation*}
\{\phi_\ell: \mathrm{Gal}_K\rightarrow \mathrm{GL}_N(\mathbb{F}_\ell)\}_{\ell\in \mathscr{P}}
\end{equation*}
is said to be \textit{strictly compatible} if $\{\phi_\ell\}$ is continuous, semisimple, and satisfies the following conditions:
\begin{enumerate}
\item[(i)] There is a finite subset $S\subset\Sigma_K$ such that $\phi_\ell$ is \emph{unramified} outside $S_\ell:=S\cup\{v\in\Sigma_K: v|\ell \}$ for all $\ell$,
\item[(ii)] For any $\ell_1,\ell_2\in\mathscr{P}$ and any $\bar{v}\in\Sigma_{\bar{K}}$ extending some $v\in\Sigma_K\backslash(S_{\ell_1}\cup S_{\ell_2})$, the characteristic polynomials of $\phi_{\ell_1}(\mathrm{Frob}_{\bar{v}})$ and $\phi_{\ell_2}(\mathrm{Frob}_{\bar{v}})$ are the reductions mod $\ell_1$ and mod $\ell_2$ of some polynomial $P_v(x)\in\mathbb{Q}[x]$ depending only on $v$.\\
\end{enumerate}
\end{definition}

Let $\rho:\mathbf{G}\rightarrow \GL_{N,F}$ be a faithful representation of a rank $r$ reductive algebraic group $\mathbf{G}$ defined over field $F$. We define in the beginning of $\mathsection2$ \emph{the formal character} of $\rho$ as an element of quotient set $\GL_r(\Z)\backslash \Z[\Z^r]$. Here $\Z[\Z^r]$ is the free abelian group generated by $\Z^r$ and $\GL_r(\Z)$ acts naturally on $\Z[\Z^r]$. This allows us to define what is meant by two representations having \emph{the same formal character} (see Definition \ref{fc}') and the formal character is \emph{bounded by a constant} $C>0$ (see Definition \ref{bfc},\ref{bfc}'). Let $\{\phi_\ell\}$ be a strictly compatible system of mod $\ell$ representations arising from \'etale cohomology (Definition \ref{arise},\ref{comsys}), this paper studies $\ell$-independence of mod $\ell$ Galois images $\bar{\Gamma}_\ell$ for all sufficiently large $\ell$. Let $\mathfrak{g}$ be a Lie type. We define \emph{total $\ell$-rank} $\rk_\ell\bar\Gamma$ and \emph{$\mathfrak{g}$-type $\ell$-rank} $\rk_\ell^{\mathfrak{g}}\bar\Gamma$ of a finite group $\bar\Gamma$ in $\mathsection3.3$ Definition \ref{rank}. The main results are as follows.

\begin{customthm}{A}\label{main}(Main theorem) \textit{Let $K$ be a number field and $\{\phi_\ell:\mathrm{Gal}_K\rightarrow\mathrm{GL}_N(\mathbb{F}_\ell)\}_{\ell\in\mathscr{P}}$ a strictly
compatible system of mod $\ell$ Galois representations  arising from \'etale cohomology (Definition \ref{arise},\ref{comsys}). There exists a finite normal extension $L$ of $K$ such that if we denote $\phi_\ell(\mathrm{Gal}_K)$ and $\phi_\ell(\mathrm{Gal}_L)$ by respectively $\bar\Gamma_\ell$ and $\bar{\gamma}_\ell$ for all $\ell$, and let $\bar{\mathbf{S}}_\ell\subset\GL_{N,\F_\ell}$ be the connected $\mathbb{F}_\ell$-semisimple subgroup associated to $\bar{\gamma}_\ell$ (or $\bar\Gamma_\ell$) by Nori's theory for $\ell\gg1$, then the following hold for $\ell\gg1:$
\begin{enumerate}
\item[(i)] The formal character of $\bar{\mathbf{S}}_\ell\hookrightarrow \mathrm{GL}_{N,\mathbb{F}_\ell}$ is independent of $\ell$ (Definition \ref{fc}') and is equal to the formal character of  $(\mathbf{G}_\ell^\circ)^{\mathrm{der}}\hookrightarrow \mathrm{GL}_{N,\Q_\ell}$, where $(\mathbf{G}_\ell^\circ)^{\mathrm{der}}$ is the derived group of the identity component of $\mathbf{G}_\ell$, the algebraic monodromy group of the semi-simplified representation $\Phi_\ell^{\ss}$.
\item[(ii)] The composition factors of $\bar{\gamma}_\ell$ and $\bar{\mathbf{S}}_\ell(\mathbb{F}_\ell)$ are identical modulo cyclic groups. Therefore, the composition factors of $\bar{\gamma}_\ell$ are finite simple groups of Lie type of characteristic $\ell$ and cyclic groups. 
\end{enumerate}}
\end{customthm}

\begin{customcor}{B}\label{cor} \textit{Let $\bar{\mathbf{S}}_\ell$ be defined as above, then the following hold for $\ell\gg1:$
\begin{enumerate}
\item[(i)] The total $\ell$-rank $\rk_\ell\bar\Gamma_\ell$ of $\bar\Gamma_\ell$ (Definition \ref{rank}) is equal to the rank of $\bar{\mathbf{S}}_\ell$ and is therefore independent of $\ell$.
\item[(ii)] The $A_n$-type $\ell$-rank $\rk_\ell^{A_n}\bar\Gamma_\ell$ of $\bar\Gamma_\ell$ (Definition \ref{rank}) for $n\in\N\backslash\{1,2,3,4,5,7,8\}$ and 
the parity of $(\rk_\ell^{A_4}\bar\Gamma_\ell)/4$ are independent of $\ell$. 
\end{enumerate}}
\end{customcor}

\begin{customrem}{1.1}\label{max} As an application of the main results, we prove in \cite{HL} that $\Phi_\ell(\mathrm{Gal}_K)$, the $\ell$-adic Galois image  arising from \'etale cohomology has certain maximality inside the algebraic monodromy group $\mathbf{G}_\ell$ if $\ell$ is sufficiently large and $\mathbf{G}_\ell$ is of type A. This generalizes Serre's open image theorem on non-CM elliptic curves \cite{S72}.\end{customrem}

\begin{customrem}{1.2}\label{iota} For any field $F$, define $\iota$ to be the involution of $\GL_{N,F}$ that sends  $A$ to $(A^t)^{-1}$. If $\Gamma$ is a subgroup of $\GL_N(F)$, then $\Gamma$ is semisimple on $F^N$ if and only if  $\iota(\Gamma)$ is semisimple on $F^N$. If $\phi_\ell$ is the mod $\ell$ Galois representation arising from the dual representation $H^i_{\mathrm{\acute{e}t}}(X_{\bar{K}},\Q_\ell)^\vee$ (Definition \ref{arise}), then the mod $\ell$ representation arising from $H^i_{\mathrm{\acute{e}t}}(X_{\bar{K}},\Q_\ell)$ is $\iota\circ\phi_\ell$ under suitable basis by Brauer-Nesbitt \cite[Theorem 30.16]{CR}. Since $\iota$ is an automorphism of $\GL_{N}$, it suffices to prove Theorem \ref{main} by considering only the dual mod $\ell$ system $\{\phi_\ell\}$ and Galois images $\{\bar{\Gamma}_\ell\}$. Let $\phi_{\bar{v}}$ be the restriction of $\phi_\ell$ to inertia subgroup $I_{\bar{v}}$ such that $\bar{v}\in\Sigma_{\bar{K}}$ divides $\ell$. The reason for choosing the dual system is that the characters of $\phi_{\bar{v}}^{\ss}$ have  bounded exponents in the sense of Definition \ref{restrict} for $\ell\gg1$ by Serre's tame inertia conjecture proved by Caruso \cite{Car} (see Theorem \ref{231}). Such boundedness makes our arguments simpler.\\
 \end{customrem}

This paper can be considered as ``mod $\ell$'' version of \cite{Hui} in which we studied $\ell$-independence of monodromy groups of any compatible system of $\ell$-adic representations by the theory of abelian $\ell$-adic representation \cite{Sbk} and the representation theory of complex semisimple Lie algebra. The main difference between \cite{Hui} and this paper is that 
you get nothing new for 
considering monodromy groups of mod $\ell$ Galois images because they are just finite groups. 
The strategy in this paper is to first construct for each $\ell\gg1$ a connected $\mathbb{F}_\ell$-reductive subgroup $\bar{\mathbf{G}}_\ell\subset\mathrm{GL}_{N,\mathbb{F}_\ell}$ with bounded formal characters (Definition \ref{bfc},\ref{bfc}') such that $[\bar{\Gamma}_\ell:\bar{\Gamma}_\ell\cap \bar{\mathbf{G}}_\ell(\mathbb F_\ell)]$ and $[\bar{\mathbf{G}}_\ell(\mathbb F_\ell):\bar{\Gamma}_\ell\cap \bar{\mathbf{G}}_\ell(\mathbb F_\ell)]$ are both uniformly bounded (Theorem \ref{redenv}). The idea to construct such $\bar{\mathbf{G}}_\ell$ was due to Serre \cite{SL2} where he considered the Galois action on the $\ell$-torsion points of abelian varieties $A$ without complex multiplication. In Serre's case, the semisimple part $\bar{\mathbf{S}}_\ell$ of $\bar{\mathbf{G}}_\ell$ is constructed by Nori's theory \cite{Nori} and the center $\bar{\mathbf{C}}_\ell$ of $\bar{\mathbf{G}}_\ell$ is the mod $\ell$ reduction of some $\Q$-diagonalizable group $\mathbf{C}_{\Q}$ which is a subgroup of the centralizer of monodromy $\mathbf{G}_\ell$ in $\GL_{N,\Q_\ell}$. Hence, $\{\bar{\mathbf{G}}_\ell\subset \GL_{N,\F_\ell}\}_\ell$ has bounded formal characters. The construction of $\mathbf{C}_{\Q}$ uses the abelian theory of $\ell$-adic representations \cite{Sbk} and the Tate conjecture for abelian variety (proved by Faltings \cite{F}) which relates the endomorphism ring of $A$ and the commutant of Galois image $\Gamma_\ell$ in $\End_N(\Q_\ell)$. Although we don't have the luxury of the Tate conjecture for \'etale cohomology in general, it is still possible to construct reductive $\bar{\mathbf{G}}_\ell\subset \GL_{N,\F_\ell}$ with the above conditions for $\ell\gg1$ by Nori's theory, tame inertia tori \cite{SL2}, and Serre's tame inertia conjecture (proved by Caruso \cite{Car}). The constructions of these algebraic envelopes $\bar{\mathbf{G}}_\ell$  of $\bar{\Gamma}_\ell$ (see Definition \ref{env}) are accomplished in $\mathsection2$. Once these nice envelopes are ready, we can use the techniques in \cite[$\mathsection3$]{Hui} to prove that the formal character (Definition \ref{fc}) of the semisimple part $\bar{\mathbf{S}}_\ell\hookrightarrow \mathrm{GL}_{N,\mathbb F_\ell}$ is independent of $\ell\gg1$ (Theorem \ref{main}). The number of $A_n$ factors of $\bar{\mathbf{S}}_\ell$ for large $n$ are then independent of $\ell$ for all $\ell\gg1$ by \cite[Theorem 2.19]{Hui}. Since the group of $\mathbb F_\ell$-rational points of $\bar{\mathbf{G}}_\ell$ is commensurate to the Galois image $\bar{\Gamma}_\ell$, one deduces $\ell$-independence results on the number of Lie type composition factors of characteristic $\ell$ of $\bar{\Gamma}_\ell$ for $\ell\gg1$ (Corollary \ref{cor}). Section $3$ is devoted to the proofs of Theorem \ref{main} and Corollary \ref{cor}. 
The following summarizes the symbols that are frequently used within the text. Groups inside $\GL_{N,F}$ with $\mathrm{char}F>0$ have their symbols over-lined and should not be confused with base change to an algebraic closure.\\

\begin{tabular}{ll}
$\mathrm{Gal}_F$ & the absolute Galois group of field $F$\\
$K$, $L$ & number fields\\
$\bar v$ & a valuation of $\bar K$ that divides prime $\ell$\\
$I_{\bar v}$ & the inertia subgroup of $\mathrm{Gal}_K$ at valuation $\bar v$\\ 
$U_\ell$, $V_\ell$, $W_\ell$ ($\bar U_\ell$, $\bar V_\ell$, $\bar W_\ell$), $\ldots$ & vector spaces defined over $\F_\ell$ (over $\bar\F_\ell$)\\
$\bar{\Gamma}_\ell$, $\bar{\gamma}_\ell$, $\bar{\Omega}_\ell$, $\bar{\Omega}_{\bar v}$, $\ldots$ & finite subgroups of $\GL_N(\F_\ell)$\\
$\mathbf{G}_\ell$, $\mathbf{T}_\ell$, $\ldots$ & algebraic subgroups of $\GL_{N,\Q_\ell}$\\
$\bar{\mathbf{G}}_\ell$, $\bar{\mathbf{S}}_\ell$, $\bar{\mathbf{N}}_\ell$, $\bar{\mathbf{I}}_\ell$, $\bar{\mathbf{I}}_{\bar v}$, $\ldots$  & algebraic subgroups of $\GL_{N,\F_\ell}$\\
$\Phi_\ell$, $\Psi_\ell$, $\Theta_\ell$, $\ldots$ & representations over $\Q_\ell$\\
$\phi_\ell$, $\psi_\ell$, $\mu_\ell$, $t_\ell$, $\rho_{\bar v}$, $f_{\bar v}$, $w_{\bar v}$, $\ldots$ & representations over $\F_\ell$ \\
$\rho^{\ss}$ & the semi-simplification of representation $\rho$\\
$\rho^{\vee}$ & the dual representation of representation $\rho$
\end{tabular}

\section{Algebraic envelope $\bar{\mathbf{G}}_\ell$}

We define \emph{formal character} and prove some related propositions before stating the main result (Theorem \ref{redenv}) of this section. Let $\rho:\mathbf{G}\rightarrow\GL_{N,F}$ be a faithful representation of a rank $r$ reductive algebraic group $\mathbf{G}$ defined over field $F$. Choose a maximal torus $\mathbf{T}$ of $\mathbf{G}$ and denote the character group of $\mathbf{T}$ by $\X$. Let $\{w_1,w_2,...,w_N\}\subset \X$ be the \emph{multiset} of weights of $\rho|_{\mathbf{T}}$ over $\bar{F}$ and choose an isomorphism $\X\cong\Z^r$. Then the image of $w_1+w_2+\cdots+w_N\in \Z[\X]\cong\Z[\Z^r]$ in the quotient set $\GL(\X)\backslash \Z[\X]\cong \GL_r(\Z)\backslash \Z[\Z^r]$ is independent of the choices of maximal torus $\mathbf{T}$ and isomorphism $\X\cong\Z^r$. 

\begin{definition}\label{fc} Let $\rho$ be as above. \emph{The formal character} of $\rho$ is defined to be the image of $w_1+w_2+\cdots+w_N\in\Z[\Z^r]$ in $\GL_r(\Z)\backslash \Z[\Z^r]$.\end{definition}

This definition of formal character is more general than the one in \cite[$\mathsection2.1$]{Hui}. It allows us to compare formal characters of two $N$-dimensional faithful representations $\rho_1:\mathbf{G}_1\rightarrow\GL_{N,F_1}$ and $\rho_2:\mathbf{G}_2\rightarrow\GL_{N,F_2}$ over different fields whenever  $\mathbf{G}_1$ and $\mathbf{G}_2$ have the same rank. Let $\mathbb{G}_{m}^{N}$ be the diagonal subgroup of $\GL_{N}$. Every character $\chi$ of $\mathbb{G}_{m}^{N}$ can be expressed uniquely as $x_1^{m_1}x_2^{m_2}\cdots x_N^{m_N}$, a product of powers of \emph{standard characters} $\{x_1,x_2,...,x_N\}$, where $x_i$ maps $(a_1,...,a_N)\in  \mathbb{G}_{m}^{N}$ to $a_i$ for all $i$. The following  proposition (definition) is particularly useful.

\begin{proposition}\label{p201}(Definition \ref{fc}') Let $\rho_1$ and $\rho_2$ be as above. If $\mathbf{T}_1\subset \mathbf{G}_1$ and $\mathbf{T}_2\subset\mathbf{G}_2$ are maximal tori. The following conditions are equivalent:
\begin{enumerate}
\item[(i)] Representations $\rho_1$ and $\rho_2$ have the same formal character.
\item[(ii)] Tori $\rho_1(\mathbf{T}_1)$ and $\rho_2(\mathbf{T}_2)$ are respectively conjugate (in $\GL_{N,\bar{F}_1}$ and $\GL_{N,\bar{F}_2}$) to some subtori $\mathbf{D}_1$ and $\mathbf{D}_2$  of the diagonal subgroup $\mathbb{G}_{m}^N\subset \GL_N$ such that $\mathbf{D}_1$ and $\mathbf{D}_2$ are annihilated by the same set of characters of $\mathbb{G}_{m}^N$.
\end{enumerate}
Hence, formal characters of $N$-dimensional faithful representations are in bijective correspondence with subtori in $\mathbb{G}_{m}^{N}$ up to natural action of permutation group $\mathrm{Perm}(N)$ of $N$ letters on $\mathbb{G}_{m}^{N}$.\end{proposition}

\begin{proof}
Assume $\mathbf{T}_j= \mathbb{G}_{m,\bar{F}_j}^r$ and $\rho_j(\mathbf{T}_j)\subset \mathbb{G}_{m,\bar{F}_j}^N\subset \GL_{N,\bar{F_j}}$ from now on by base change to algebraic closure of $F_j$ and diagonalizations for $j=1,2$. Suppose (i) holds, then by taking an automorphism of the character group of $\mathbf{T}_1$ and a permutation of coordinates of $\mathbb{G}_{m}^{N}$ we obtain 
\begin{equation*}
x_i\circ\rho_1=x_i\circ\rho_2
\end{equation*}
for all standard character $x_i$ of $\mathbb{G}_{m}^{N}$ if we identify the character groups of $\mathbb{G}_{m,\bar{F}_1}^r$ and $\mathbb{G}_{m,\bar{F}_2}^r$ naturally. This implies the set of characters of $\mathbb{G}_{m}^{N}$ that annihilate $\mathbf{D}_j:=\rho_j(\mathbf{T}_j)$ for $j=1,2$ are equal which is (ii). Suppose (ii) holds, it suffices to consider the case that $\rho_1$ and $\rho_2$ are standard representations (inclusions) since $\rho:\mathbf{G}\rightarrow \GL_{N,F}$ and $\rho(\mathbf{G})\subset\GL_{N,F}$ always have the same formal character. Condition (ii) implies that there exists an automorphism of $\mathbb{G}_{m}^{N}$ such that  
\begin{equation*}
\mathbf{D}_j=\{(a_1,...,a_N)\in \mathbb{G}_{m}^{N}:~a_1=a_2=\cdots=a_{N-r}=1\}
\end{equation*}
for $j=1,2$ because $\mathbf{D}_1$ and $\mathbf{D}_2$ are connected. Then (i) follows easily.

Let $\rho:\mathbf{T}\rightarrow \GL_{N,\bar{F}}$ be a representation of a torus $\mathbf{T}$. Since the set of weights of $\rho$ is obtained by diagonalizing $\rho(\mathbf{T})$ and is independent of diagonalizations, subtori of  $\mathbb{G}_{m}^{N}$ that are conjugate to $\rho(\mathbf{T})$ only differ by a permutation of $N$ coordinates. Therefore, the map from formal characters of $N$-dimensional faithful representations to subtori of $\mathbb{G}_{m}^{N}$ modulo action of $\mathrm{Perm}(N)$ is well defined. Since the equivalence of (i) and (ii) implies injectivity and any subtorus of $\mathbb{G}_{m}^{N}$ is the formal character of the standard representation of the subtorus, the map is a bijective correspondence.
\end{proof}

\noindent \textbf{Examples}: Denote standard representation and faithful representation by respectively Std and $\rho$. Below are some pair of representations that have the same formal character:
\begin{enumerate}
\item[(i)] $(\GL_{2,\Q_\ell},\mathrm{Std})$ and $(\GL_{2,\F_\ell},\mathrm{Std})$;
\item[(ii)] $(\mathbf{G},\rho)$ and $(\mathbf{H},\rho|_{\mathbf{H}})$ if $\mathbf{H}$ is a reductive subgroup of $\mathbf{G}$ of same rank;
\item[(iii)] $(\mathbf{G},\rho)$ and $(\mathbf{G},\rho^\vee)$; 
\item[(iv)] $(\mathbf{G},\rho)$ and $(\rho(\mathbf{G}),\mathrm{Std})$.\\
\end{enumerate}

\begin{definition}\label{bfc} The formal character of $\rho$ is said to be \emph{bounded by a constant} $C>0$ if there exists an isomorphism $\X\cong\Z^r$ such that the coefficients of the images of weights $w_1,w_2,...,w_N\in\X$ in $\Z^r$ have absolute values bounded by $C$.\\
Let $N$ be a fixed integer and $\{\rho_i:\mathbf{G}_i\rightarrow\GL_{N_i,F_i}\}_{i\in I}$ a family of faithful representations of reductive groups such that
$N_i\leq N$ for all $i\in I$. The family is said to have \emph{bounded formal characters} if the formal character of $\rho_i$ is bounded by some constant $C>0$ for all $i\in I$. \end{definition}

\begin{remark}\label{finitefc} Let $\{\rho_i\}_{i\in I}$ be a family of representations in Definition \ref{bfc} having bounded formal characters. Then the number of distinct formal characters arising from the family is finite. \end{remark}

Let $\chi=x_1^{m_1}x_2^{m_2}\cdots x_N^{m_N}$ be a character of $\mathbb{G}_{m}^{N}$ expressed as products of standard characters. We call multiset $\{m_1,...,m_N\}$ \emph{the exponents} of $\chi$ and say \emph{the exponents are bounded by} $C>0$ if $|m_i|<C$ for all $1\leq i\leq N$. The following characterization of Definition \ref{bfc} is very useful in this paper.

\begin{proposition}\label{p203}(Definition \ref{bfc}') Let $\{\rho_i\}_{i\in I}$ be a family of  faithful representations of reductive $\mathbf{G}_i$ such that $\rho_i$ is $N_i$-dimensional and $N_i\leq N$ for all $i\in I$. Choose a maximal torus $\mathbf{T}_i$ of $\mathbf{G}_i$ for each $i\in I$. The following conditions are equivalent:
\begin{enumerate}
\item[(i)] The family has bounded formal characters.
\item[(ii)] For any $i\in I$ and any subtorus $\mathbf{D}_i$ of the diagonal subgroup $\mathbb{G}_{m}^{N_i}\subset \GL_{N_i}$ that is conjugate (in $\GL_{N_i,\bar{F_i}}$) to $\rho_i(\mathbf{T}_i)$, one can choose a set $R_i$ of characters of $\mathbb{G}_{m}^{N_i}$  such that the common kernel of $R_i$ is $\mathbf{D}_i$ and the exponents of characters in $R_i$ are bounded by a constant independent of $i\in I$.
\end{enumerate}
\end{proposition}

\begin{proof} It follows easily from Definition \ref{bfc}, the bijective correspondence in Proposition \ref{p201}, and Remark \ref{finitefc}.\end{proof}

\begin{proposition}\label{p204} Let $\{\rho_i\}_{i\in I}$ and $\{\phi_i\}_{i\in I}$ be two families of faithful representations of reductive $\mathbf{G}_i$ and $\mathbf{H}_i$ over field $F_i$ with bounded formal characters such that the target of $\rho_i$ and $\phi_i$ are both equal to $\GL_{N_i,F_i}$ and $\rho_i(\mathbf{G}_i)$ commutes with $\phi_i(\mathbf{H}_i)$ for all $i\in I$. Then the family of standard representations
\begin{equation*}
\{ \rho_i(\mathbf{G}_i)\cdot \phi_i(\mathbf{H}_i) \subset \GL_{N_i,F_i}\}_{i\in I}
\end{equation*}
also has bounded formal characters.
\end{proposition}

\begin{proof} It follows easily from Remark \ref{finitefc}, Proposition \ref{p203}, and the fact (by the commutativity hypothesis) that any maximal torus of $\rho_i(\mathbf{G}_i)\cdot \phi_i(\mathbf{H}_i)$ is generated by some maximal torus of $\rho_i(\mathbf{G}_i)$ and some maximal torus of $\phi_i(\mathbf{H}_i)$. \end{proof}

Let $\{\phi_\ell\}$ be the strictly compatible system of mod $\ell$ Galois representations arising from (Definition \ref{arise},\ref{comsys}) the dual system of $\ell$-adic representations $\{\Phi_\ell\}$. Denote the image of $\phi_\ell$ by $\bar{\Gamma}_\ell$ and the ambient space of the representation by $V_\ell\cong\mathbb{F}_\ell^N$ for each $\ell$. Each  $\bar{\Gamma}_\ell:=\phi_\ell(\mathrm{Gal}_K)$ is a subgroup of GL$_N(\F_\ell)$ for a fixed $N$. Suppose $K'$ is a finite normal extension of $K$. Since $[\phi_\ell(\mathrm{Gal}_K):\phi_\ell(\mathrm{Gal}_{K'})]\leq [K':K]$ for all $\ell$ and the restriction of $\{\phi_\ell\}$ to $\mathrm{Gal}_{K'}$ is semisimple \cite[Theorem 49.2]{CR} and satisfies the compatibility conditions (Definition \ref{comsys}), we are free to replace $K$ by $K'$ in the course of proving the main theorem. The main result of this section states that  for $\ell\gg1$, $\bar{\Gamma}_\ell$ can be approximated by some connected, reductive subgroup $\bar{\mathbf{G}}_\ell\subset\mathrm{GL}_{N,\mathbb{F}_\ell}$ with bounded formal characters (Definition \ref{bfc}').

\begin{customthm}{2.0.5}\label{redenv} \textit{Let $\{\phi_\ell\}_{\ell\in\mathscr{P}}$ be a system of mod $\ell$ Galois representations as above. There exist a finite normal extension $L$ of $K$ and a connected, $\mathbb{F}_\ell$-reductive subgroup $\bar{\mathbf{G}}_\ell$ of $\mathrm{GL}_{N,\mathbb{F}_\ell}$ for each $\ell\gg1$ such that 
\begin{enumerate}
\item[(i)] $\bar{\gamma}_\ell:=\phi_\ell(\mathrm{Gal}_L)$ is a subgroup of $\bar{\mathbf{G}}_\ell(\mathbb{F}_\ell)$ of uniformly bounded index,
\item[(ii)] the action of $\bar{\mathbf{G}}_\ell$ on $\bar{V}_\ell:=V_\ell\otimes\bar{\F}_\ell$ is semisimple,
\item[(iii)] the representations $\{\bar{\mathbf{G}}_\ell\hookrightarrow \GL_{N,\F_\ell}\}_{\ell\gg1}$ have bounded formal characters in the sense of Definition \ref{bfc}'.
\end{enumerate}}
\end{customthm}

\begin{definition}\label{env}
A system of connected reductive groups $\{\bar{\mathbf{G}}_\ell\}_{\ell\gg1}$ satisfying the conditions in the above theorem is called \textit{a system of algebraic envelopes} of $\{\bar{\Gamma}_\ell\}_{\ell\gg1}$. We say $\bar{\mathbf{G}}_\ell$ is the \emph{algebraic envelope} of $\bar{\Gamma}_\ell$ when a system of algebraic envelopes is given.
\end{definition}

We first establish in $\mathsection2.1-2.4$ essential ingredients of the proof of Theorem \ref{redenv}. Then the proof is presented in $\mathsection2.5$.

\subsection{Nori's theory}

The material in this subsection is due to Nori \cite{Nori}. Suppose $\ell>N-1$. Given a subgroup $\bar\Gamma$  of $\mathrm{GL}_N(\mathbb{F}_\ell)$, Nori's theory gives us a connected algebraic group $\bar{\mathbf{S}}_\ell$ that captures all the order $\ell$ elements of $\bar\Gamma$ if $\ell$ is bigger than a constant that only depends on $N$.

Let $\bar\Gamma[\ell]=\{x\in \bar\Gamma~|~ x^\ell=1\}$.  The normal subgroup of $\bar\Gamma$ generated by $\bar\Gamma[\ell]$ is denoted by $\bar\Gamma^+$. Define $\mathrm{exp}(x)$ and $\mathrm{log}(x)$ by
$$\mathrm{exp}(x)=\sum_{i=0}^{\ell-1}\frac{x^i}{i!}\hspace{.1in}\mathrm{and}\hspace{.1in}
\mathrm{log}(x)=-\sum_{i=1}^{\ell-1}\frac{(1-x)^i}{i}.$$
Denote by $\bar{\mathbf{S}}$ the (connected) algebraic subgroup of $\mathrm{GL}_{N,\F_\ell}$, defined over $\mathbb{F}_\ell$, generated by the one-parameter subgroups
\begin{equation*}
t\mapsto x^t:=\mathrm{exp}(t\cdot\mathrm{log}(x))
\end{equation*}
for all $x\in \bar\Gamma[\ell]$. Algebraic subgroups with the above property are said to be \textit{exponentially generated}. The theorem we need is stated below.

\begin{theorem}\label{N1}\cite[Theorem B(1), 3.6(v)]{Nori} There is a constant $c_0=c_0(N)$ such that if $\ell>c_0$ and $\bar\Gamma$ is a subgroup of $\mathrm{GL}_N(\mathbb{F}_\ell)$, then
\begin{enumerate}
\item[(i)] $\bar\Gamma^+=\bar{\mathbf{S}}(\mathbb{F}_\ell)^+$,
\item[(ii)] $\bar{\mathbf{S}}(\mathbb{F}_\ell)/\bar{\mathbf{S}}(\mathbb{F}_\ell)^+$ is a commutative group of order $\leq 2^{N-1}$.
\end{enumerate}
\end{theorem}

\begin{proposition}\label{N2} Let $\bar{\mathbf{S}}_\ell$ be the algebraic group associated to $\bar\Gamma_\ell$ by Nori's theory for all $\ell>N-1$. There is a constant $c_1=c_1(N)>c_0(N)$ that depends only on $N$ such that if $\ell>c_1$, then the following hold:
\begin{enumerate}
\item[(i)] $\bar{\mathbf{S}}_\ell$ is a connected, exponentially generated, semisimple $\mathbb{F}_\ell$-subgroup of $\mathrm{GL}_{N,\mathbb{F}_\ell}$. 
\item[(ii)] $\bar{\mathbf{S}}_\ell$ acts semi-simply on the ambient space $\bar{V}_\ell\cong \bar{\mathbb{F}}_\ell^N$.
\item[(iii)] $[\bar{\mathbf{S}}_\ell(\mathbb{F}_\ell):\bar{\mathbf{S}}_\ell(\mathbb{F}_\ell)\cap \bar{\Gamma}_\ell]\leq 2^{N-1}$.
\end{enumerate}
\end{proposition}

\begin{proof} Since $\bar{\Gamma}_\ell$ acts semi-simply on $\bar{V}_\ell$, so does $\bar{\Gamma}_\ell^+$ \cite[Theorem 49.2]{CR}. Part (ii) then follows from \cite[Theorem 24]{EHK} for some sufficiently large constant $c_1(N)$ ($>c_0(N)$) depending only on $N$, see also \cite{SL2}. Since $\ell>c_0(N)$, $\bar{\mathbf{S}}_\ell(\mathbb{F}_\ell)^+=\bar{\Gamma}_\ell^+$ (Theorem \ref{N1}) also acts semi-simply on $\bar{V}_\ell$. This implies  $\bar{\mathbf{S}}_\ell(\mathbb{F}_\ell)^+$ cannot have normal $\ell$-subgroup. If $\bar{\mathbf{S}}_\ell$ has a non-trivial  unipotent radical $\bar{\mathbf{U}}_\ell$, then $\bar{\mathbf{U}}_\ell$ is defined over $\mathbb{F}_\ell$ \cite[Proposition 14.4.5(v)]{Springbk} and $\bar{\mathbf{U}}_\ell(\mathbb{F}_\ell)$ is then a non-trivial normal $\ell$-group of $\bar{\mathbf{S}}_\ell(\mathbb{F}_\ell)^+$ which is a contradiction. Therefore $\bar{\mathbf{S}}_\ell$ is reductive. $\bar{\mathbf{S}}_\ell$ is actually semisimple since it is generated by unipotent elements $\bar{\Gamma}_\ell^+$. This proves (i).
Since $\ell>c_0(N)$,
(iii) is proved by Theorem \ref{N1}. 
\end{proof}

\begin{definition}\label{Nori}
Define the \emph{semisimple envelope}  of $\bar{\Gamma}_\ell$  for all sufficiently large $\ell$ as the connected, semisimple $\F_\ell$-algebraic group $\bar{\mathbf{S}}_\ell$ in Proposition \ref{N2}.  
\end{definition}

\begin{remark}\label{ind}If $K'$ is a finite extension of $K$, then the semisimple envelopes of $\phi_\ell(\mathrm{Gal}_{K'})$ and $\phi_\ell(\mathrm{Gal}_K)$ are identical for  $\ell\gg1$ because the order $\ell$ elements 
of the two finite groups are the same when $\ell$ is large.\end{remark}

\subsection{Characters of tame inertia group} 

Let $\rho_\ell:\mathrm{Gal}_K\rightarrow \mathrm{GL}_N(\mathbb{F}_\ell)$ be a continuous representation and $I_{\bar{v}}$ the inertia subgroup of $\mathrm{Gal}_K$ at $\bar{v}\in\Sigma_{\bar{K}}$ that divides $\ell$. Let $I_{\bar{v}}^\mathrm{w}$ be the wild inertia (normal) subgroup of $I_{\bar{v}}$ and $\rho_{\bar{v}}^{\ss}$ the semi-simplification of the restriction of $\rho_\ell$ to $I_{\bar{v}}$.   Since $\rho_\ell^{\ss}(I_{\bar{v}}^\mathrm{w})$ is an $\ell$-group and semisimple on $\F_\ell^N$, $\rho_{\bar{v}}^{\ss}(I_{\bar{v}}^\mathrm{w})=\{1\}$ and $\rho_{\bar{v}}^{\ss}$ factors through a representation of the tame inertia group $I_{\bar{v}}^\mathrm{t}:=I_{\bar{v}}/I_{\bar{v}}^\mathrm{w}$ (still denoted by $\rho_{\bar{v}}^{\ss}$):
\begin{equation*}
\rho_{\bar{v}}^{\ss}: I_{\bar{v}}^\mathrm{t}\to \mathrm{GL}_N(\mathbb{F}_\ell).
\end{equation*}
The tame inertia group $I_{\bar{v}}^\mathrm{t}$ is a projective limit of cyclic groups of order prime to $\ell$ \cite[Proposition 2]{S72}
\begin{equation*}
\theta_{\bar{v}}: I_{\bar{v}}^\mathrm{t}\stackrel{\cong}{\longrightarrow} \varprojlim_{k} \mathbb{F}_{\ell^k}^*
\end{equation*}
where the projective system is given by norm maps of finite fields of characteristic $\ell$. The isomorphism is unique up to action of $\Gal_{\F_\ell}$ on the target.

\begin{definition}\label{fund} The \textit{fundamental characters} of $I_{\bar{v}}^\mathrm{t}$ of level $d$ \cite[$\mathsection1.7$]{S72} are defined as 
\begin{equation*}
\theta_{d}^{\ell^j},~~~~~ j=0,1,...,d-1
\end{equation*}
where $\theta_{d}:I_{\bar{v}}^\mathrm{t}\stackrel{\theta_{\bar{v}}}{\longrightarrow} \varprojlim_k \mathbb{F}_{\ell^k}^*\twoheadrightarrow \mathbb{F}_{\ell^d}^*\hookrightarrow \bar{\F}_\ell^*$.\\
\end{definition}

Any continuous character $\chi: I_{\bar{v}}^\mathrm{t}\rightarrow \bar{\mathbb{F}}_\ell^*$ of $\rho_{\bar{v}}^{\ss}$ factors through a power of some $\theta_d$. Character theory says that $\mathrm{Hom}(\mathbb{F}_{\ell^d}^*,\bar{\mathbb{F}}_\ell^*)\cong\mathrm{Hom}(\mathbb{F}_{\ell^d}^*,\mathbb{C}^*)$ is cyclic generated by $\theta_d$ of order $\ell^d-1$. Therefore, $\chi$ can always be expressed as a product of fundamental characters of level $d$
\begin{equation*}
\chi=(\theta_d)^{m_0}\cdot(\theta_d^{\ell})^{m_1}\cdots(\theta_d^{\ell^{d-1}})^{m_{d-1}}
\end{equation*}

\begin{definition}\label{restrict} Let $\chi: I_{\bar{v}}^\mathrm{t}\rightarrow \bar{\mathbb{F}}_\ell^*$ be a character of $\rho_{\bar{v}}^{\ss}$ and express  $\chi$ as a product of fundamental characters of level $d$ as above.
\begin{enumerate}
\item[(i)] The product is said to be \emph{$\ell$-restricted} if $0\leq m_i\leq \ell-1$ for all $i$ and not all $m_i$ equal to $\ell-1$. It is easy to see that $\ell$-restricted expression of $\chi$ is unique.
\item[(ii)] The \emph{exponents} of $\chi$ are defined to be the 
multiset of powers $\{m_0,m_1,...,m_{d-1}\}$ in the $\ell$-restricted product.
Note that the multiset is independent of the action of $\Gal_{\F_\ell}$ on the target.
\end{enumerate}
\end{definition}

\begin{lemma}\label{221} Let $V\cong\mathbb{F}_\ell^n$ be a continuous, irreducible subrepresentation of $\rho_{\bar{v}}$, then the characters of the representation can be written as a product of fundamental characters of level $n$.\end{lemma}

\begin{proof} For simplicity, assume $\rho_{\bar{v}}$ is irreducible. The image of $I_{\bar{v}}^\mathrm{t}$ in $\mathrm{GL}(V)$ is a cyclic group of order prime to $\ell$, therefore $V$ is a $\mathbb{F}_\ell[x]/(f(x))$-module where $x$ corresponds to a generator of the cyclic image and the minimal polynomial $f(x)$ is separable. Irreducibility of $V$ implies $f(x)$ is irreducible over $\mathbb{F}_\ell$. Thus $\rho_{\bar{v}}(I_{\bar{v}}^\mathrm{t})$ is contained in a maximal subfield $F$ of $\mathrm{End}(V)$ and $\rho_{\bar{v}}:I_{\bar{v}}^\mathrm{t}\rightarrow F^*\subset \mathrm{GL}(V)$ can be written as a product of fundamental characters of level $n$ as above. On the other hand, $V$ has a structure of $F$-vector space of dimension $1$ such that the action of $\rho_{\bar{v}}(I_{\bar{v}}^\mathrm{t})\subset F^*$ is through field multiplication. By tensoring $F$ with $F$ (on the right) over $\mathbb{F}_\ell$, we obtain an $F$-isomorphism 
\[ \begin{array}{lcl}

F\otimes F\to F\oplus F\oplus\cdots\oplus F \\
~~x\otimes y\mapsto (xy,x^\ell y,...,x^{\ell^{n-1}}y)
\end{array}
\]
where $x,x^\ell,...,x^{\ell^{n-1}}$ are just conjugate of $x$ over $\mathbb{F}_\ell$. If $x\in \rho_{\bar{v}}(I_{\bar{v}}^\mathrm{t})\subset F^*$, then we see the action of $I_{\bar{v}}^\mathrm{t}$ on $V\otimes_{\mathbb{F}_\ell} F$ is a direct sum of products of fundamental characters of level $n$.\end{proof}

\subsection{Exponents of characters arising from \'etale cohomology} 

Every character $\chi$ of $\rho_{\bar{v}}^{\ss}: I_{\bar{v}}^\mathrm{t}\rightarrow \mathrm{GL}_N(\mathbb{F}_\ell)$ can be written as \begin{equation*}
\chi=(\theta_n)^{m_0}\cdot(\theta_n^{\ell})^{m_1}\cdots(\theta_n^{\ell^{n-1}})^{m_{n-1}},
\end{equation*}
a product of fundamental characters of level $n\leq N$ by Lemma \ref{221}. One would like to study the exponents $m_0,...,m_{n-1}$ (Definition \ref{restrict}) and in the case of \'etale cohomology we have the following theorem proved by Caruso \cite{Car}.

\begin{theorem}\label{231}(Serre's tame inertia conjecture)\label{conj} Let $X$ be a proper and smooth variety over a local field $K$ (a finite extension of $\Q_\ell$) with semi-stable reduction over $\mathscr{O}_{K}$, the ring of integers of $K$ and $i$ an integer. The Galois group $\mathrm{Gal}_K$ acts on $H^i_{\mathrm{\acute{e}t}}(X_{\bar{K}},\mathbb{Z}/\ell\mathbb{Z})^\vee$, the $\mathbb{F}_\ell$-dual of the $i$th cohomology group with $\mathbb{Z}/\ell\mathbb{Z}$ coefficients. If we restrict the representation to the inertia group of $\mathrm{Gal}_K$, then the exponents of the characters of the tame inertia group on any Jordan-Holder quotient of $H^i_{\mathrm{\acute{e}t}}(X_{\bar{K}},\mathbb{Z}/\ell\mathbb{Z})^\vee$ are between $0$ and $ei$ where $e$ 
is the ramification index of $K/\Q_\ell$.\\\end{theorem}

We now relate our mod $\ell$ Galois representation $\phi_\ell$ to representation $H^i_{\mathrm{\acute{e}t}}(X_{\bar{K}},\mathbb{Z}/\ell\mathbb{Z})^\vee$ in Theorem \ref{231}. Cohomology group $H^i_{\mathrm{\acute{e}t}}(X_{\bar{K}},\mathbb{Z}_\ell)$ is a finitely generated, free $\mathbb{Z}_\ell$-module \cite{Gab} for $\ell\gg1$:
\begin{equation*}
H^i_{\mathrm{\acute{e}t}}(X_{\bar{K}},\mathbb{Z}_\ell)\cong \mathbb{Z}_\ell\oplus\cdots\oplus\mathbb{Z}_\ell.
\end{equation*}
Reduction mod $\ell$ gives 
\begin{equation*}
H^i_{\mathrm{\acute{e}t}}(X_{\bar{K}},\mathbb{Z}_\ell)\otimes \F_\ell= \mathbb{Z}/\ell\mathbb Z\oplus\cdots\oplus\mathbb{Z}/\ell\mathbb Z
\end{equation*}
and the semi-simplification of $H^i_{\mathrm{\acute{e}t}}(X_{\bar{K}},\mathbb{Z}_\ell)\otimes \F_\ell$ is then
isomorphic to the semi-simplification of a mod $\ell$ reduction of 
$\ell$-adic representation $H^i_{\mathrm{\acute{e}t}}(X_{\bar{K}},\Q_\ell)$ by Brauer-Nesbitt \cite[Theorem 30.16]{CR}. Since the sequence
\begin{equation*} H^i_{\mathrm{\acute{e}t}}(X_{\bar{K}},\mathbb{Z}_\ell)\stackrel{\ell}{\rightarrow} H^i_{\mathrm{\acute{e}t}}(X_{\bar{K}},\mathbb{Z}_\ell)\rightarrow H^i_{\mathrm{\acute{e}t}}(X_{\bar{K}},\mathbb{Z}/\ell\mathbb{Z})\end{equation*}
 is exact \cite[Theorem 19.2]{Milne}, $H^i_{\mathrm{\acute{e}t}}(X_{\bar{K}},\mathbb{Z}_\ell)\otimes \F_\ell$ is isomorphic to $H^i_{\mathrm{\acute{e}t}}(X_{\bar{K}},\mathbb{Z}/\ell\mathbb{Z})$. Recall $V_\ell$ is the semi-simplification of a mod $\ell$ reduction of 
$H^i_{\mathrm{\acute{e}t}}(X_{\bar{K}},\Q_\ell)^\vee$.
Thus, we conclude that

\begin{proposition}\label{232} For all sufficiently large $\ell$, $H^i_{\mathrm{\acute{e}t}}(X_{\bar{K}},\mathbb{Z}_\ell)\otimes \F_\ell$ is isomorphic to $H^i_{\mathrm{\acute{e}t}}(X_{\bar{K}},\mathbb{Z}/\ell\mathbb{Z})$ and the semi-simplification of $H^i_{\mathrm{\acute{e}t}}(X_{\bar{K}},\mathbb{Z}/\ell\mathbb{Z})$ is $V_\ell^\vee$.\end{proposition}

The following theorem is the main result of this subsection.

\begin{theorem}\label{233} Let $K$ be a number field. Let $\phi_\ell: \mathrm{Gal}_K\rightarrow \mathrm{GL}(V_\ell)\cong\mathrm{GL}_N(\mathbb{F}_\ell)$ be the mod $\ell$ Galois representation arising from \'etale cohomology group $H^i_{\mathrm{\acute{e}t}}(X_{\bar{K}},\Q_\ell)^\vee$ for sufficiently large $\ell$. If we restrict $\phi_\ell$ to the inertia group $I_{\bar{v}}$ of a valuation $\bar{v}|\ell$ of $\bar{K}$ and semi-simplify the representation, then every character $\chi$ of the representation can be written as
\begin{equation*}
\chi=(\theta_{N!})^{m_0}\cdot(\theta_{N!}^{\ell})^{m_1}\cdots(\theta_{N!}^{\ell^{N!-1}})^{m_{N!-1}}
\end{equation*}
a product of fundamental characters of level $N!$ with exponents (Definition \ref{restrict}) $m_0,...,m_{N!-1}$ (depending on $\ell$) belonging to $[0,ei]$ where $e$ is the ramification index of $K_v/\Q_\ell$, $v=\bar{v}|_K$, and $K_v$ is the completion of $K$ with respect to $v$.
\end{theorem}

\begin{proof} Proposition \ref{232} implies that if $\ell$ is sufficiently large, then Galois representations $V_\ell=(V_\ell^\vee)^\vee$ and $(H^i_{\mathrm{\acute{e}t}}(X_{\bar{K}},\mathbb{Z}/\ell\mathbb{Z})^\vee)^{\ss}$ are isomorphic.
 Let $\chi$ be a character of $I_{\bar{v}}^\mathrm{t}$ given by the semi-simplification of the restriction of $V_\ell$ 
to inertia subgroup $I_{\bar{v}}$.
By Theorem \ref{231},  $\chi$ can be written as
\begin{equation*}
\chi=(\theta_d)^{m_0}\cdot(\theta_d^{\ell})^{m_1}\cdots(\theta_d^{\ell^{d-1}})^{m_{d-1}},
\end{equation*}
a product of fundamental characters of level $d$ ($\leq N$ by Lemma \ref{221}) with exponents $m_0,...,m_{d-1}$ belonging to $[0,ei]$ where $e$ is the ramification index of $K_v/\Q_\ell$. Since $d$ divides $N!$, $\theta_{N!}$ factors through $\chi$. Consider the norm map $\mathrm{Nm}:\mathbb{F}_{\ell^{N!}}^*\to \mathbb{F}_{\ell^d}^*$
\[ \begin{array}{lcl}
x\mapsto x\cdot x^{\ell^d}\cdot x^{\ell^{2d}} \cdots x^{\ell^{(N!-d)}}.
\end{array}
\]
Then we obtain a product of fundamental characters of level $N!$
\begin{equation*}
\chi=(\mathrm{Nm} \circ\theta_{N!})^{{m_0}+m_1\ell+\cdots+ m_{d-1}\ell^{d-1}}
\end{equation*}
\begin{equation*}
=(\theta_{N!})^{s_0}\cdot(\theta_{N!}^{\ell})^{s_1}\cdots(\theta_{N!}^{\ell^{N!-1}})^{s_{N!-1}}
\end{equation*}
with exponents $s_0,...,s_{N!-1}$ belonging to $[0,ei]$.\end{proof}

\subsection{Tame inertia tori and rigidity}

Tame inertia tori were considered by Serre when he studied Galois action on $\ell$-torsion points of abelian varieties without complex multiplication \cite{SL2}. He observed that these tori have certain rigidity which will be explained in this subsection.

Assume $\ell> N-1$ as in $\mathsection2.1$. Since every non-trivial element of every $\ell$-Sylow subgroup of $\bar\Gamma_\ell$ is of order $\ell$ and $\bar{\Gamma}_\ell^+$ is contained in $\bar{\mathbf{S}}_\ell(\mathbb{F}_\ell)$ by Theorem \ref{N1}(i), index $[\bar{\Gamma}_\ell:\bar{\Gamma}_\ell\cap \bar{\mathbf{S}}_\ell(\mathbb{F}_\ell)]$ is prime to $\ell$.  Let $\bar{\mathbf{N}}_\ell$ be the normalizer of $\bar{\mathbf{S}}_\ell$ in $\mathrm{GL}_{N,\mathbb{F}_\ell}$; clearly $\bar{\Gamma}_\ell\subset \bar{\mathbf{N}}_\ell$.

\begin{theorem}\label{241}\cite[$\mathsection1$ Theorem]{SL2} There are constants $c_2=c_2(N)$ and $c_3=c_3(N)$ such that if $\ell>c_2$, $\bar{\mathbf{S}}_\ell\subset\mathrm{GL}_{N,\mathbb{F}_\ell}$ is an exponentially generated semisimple algebraic group defined over $\mathbb{F}_\ell$, and the action on $\bar{V}_\ell\cong\bar{\mathbb{F}}_\ell^N$ is semisimple. If $W_\ell$ is the $\mathbb{F}_\ell$-subspace of 
\begin{equation*}
U_\ell:=\bigoplus_{i=1}^{ c_3}(\otimes^i V_\ell)
\end{equation*} fixed by $\bar{\mathbf{S}}_\ell$, then $t_\ell:\bar{\mathbf{N}}_\ell/\bar{\mathbf{S}}_\ell\rightarrow \mathrm{GL}_{W_\ell}$ is an $\mathbb{F}_\ell$-embedding. Moreover, if $x\notin \bar{\mathbf{S}}_\ell$, then there is an element of $\bar{W}_\ell$ that is not fixed by $x$.
\end{theorem}

By Theorem \ref{241}, $\bar{\Gamma}_\ell/(\bar{\Gamma}_\ell\cap \bar{\mathbf{S}}_\ell(\mathbb{F}_\ell))$ embeds in $\mathrm{GL}(W_\ell)$ with dim$(W_\ell)\leq c_4=c_4(N)$ uniformly for some integer $c_4$. Theorem \ref{242} below is the main result of this subsection. 

\begin{definition}\label{mu} Define $\mu_\ell:\mathrm{Gal}_K\rightarrow \GL(W_\ell)$ to be the composition $t_\ell\circ\phi_\ell$ for each $\ell$ and $\bar{\Omega}_\ell$ to be the image $\mu_\ell$, where $t_\ell$ is defined in Theorem \ref{241}.
\end{definition}

\begin{theorem}\label{242} Let $\bar{\mathbf{I}}_\ell$ be the algebraic group generated by a set of tame inertia tori $\bar{\mathbf{I}}_{\bar{v}}$ (Definition \ref{tori}) for $\ell\gg1$. There exist constant $c_8=c_8(N)$ and a finite normal field extension $L/K$ such that if $\ell\gg1$, then $\bar{\mathbf{I}}_\ell$ is a torus, called the inertia torus at $\ell$, and $\mu_\ell(\mathrm{Gal}_L)\subset \bar{\Omega}_\ell$ is a subgroup of $\bar{\mathbf{I}}_\ell(\mathbb{F}_\ell)$ such that
\begin{enumerate}
\item[(i)] $\{\bar{\mathbf{I}}_\ell\hookrightarrow \GL_{W_\ell}\}_{\ell\gg1}$ have bounded formal characters (Definition \ref{bfc}'),
\item[(ii)] $[\bar{\mathbf{I}}_\ell(\mathbb{F}_\ell):\mu_\ell(\mathrm{Gal}_L)]$ is bounded by $c_8$.\\
\end{enumerate} \end{theorem}

\begin{theorem}\label{Jordan}\cite[Jordan's theorem on finite linear groups]{Jordan} For every $n$ there exists a constant $J(n)$ such that any finite subgroup
of $\mathrm{GL}_n$ over a field of characteristic zero possesses an abelian normal subgroup
of index $\leq J(n)$.\end{theorem}

The order of $\bar{\Omega}_\ell$ is prime to $\ell$. $\bar{\Omega}_\ell$ can thus be lifted to a subgroup of GL$_{N'}(\mathbb{C})$ such that $N'$ only depends on $N$. Theorem \ref{Jordan} (Jordan) then says that $\bar{\Omega}_\ell$ has a abelian normal subgroup $\bar{J}_\ell$ of index less than a constant $c_5=c_5(N):=J(N')$ depends on $N'$. Since $N'$ depends on $N$, we have $[\bar{\Omega}_\ell:\bar{J}_\ell]\leq c_5$. If $\bar{v}$ divides $\ell$, then the action of the inertia group $I_{\bar{v}}$ on $W_\ell$ is semisimple because $|\bar{\Omega}_\ell|$ is prime to $\ell$. Since dim$(W_\ell)|c_4!$ We obtain
\begin{equation*}
\mu_\ell:I_{\bar{v}}^\mathrm{t}\stackrel{\theta_{c_4!}}{\twoheadrightarrow}\mathbb{F}_{\ell^{c_4!}}^*\rightarrow \mathrm{GL}(W_\ell).
\end{equation*}
By Theorem \ref{233} and $W_\ell$ in Theorem \ref{241}, there exist $c_6=c_6(N)\geq0$ such that if $\chi$ is a character, then $\chi$ can be written as a product of fundamental characters of level $c_4!$ 
\begin{equation*}
\chi=(\theta_{c_4!})^{m_0}\cdot(\theta_{c_4!}^{\ell})^{m_1}\cdots(\theta_{c_4!}^{\ell^{c_4!-1}})^{m_{c_4!-1}}
\end{equation*}
with exponents $m_0,...,m_{c_4!-1}$ belonging to $[0,c_6]$ for all $\ell\gg1$. Therefore, we make the following definition.

\begin{definition}\label{tori} Denote field $\mathbb{F}_{\ell^{c_4!}}$ by $\mathbb{E}_\ell$ for all $\ell$. This gives a homomorphism
\begin{equation*}
f_{\bar{v}}:\mathbb{E}_\ell^*\rightarrow \mathrm{GL}(W_\ell)
\end{equation*}
 if $\ell>c_6(N)+1$. Let $\bar{\mathbf{E}}_\ell$ denote $\mathrm{Res}_{\mathbb{E}_\ell/\mathbb{F}_\ell}(\mathbb{G}_m)$ (Weil restriction of scalars) for all $\ell$. We have $\bar{\mathbf{E}}_\ell(\F_\ell)=\mathbb{E}_\ell^*$. Then $f_{\bar{v}}$ extends uniquely \cite[$\mathsection3$]{Hall} to an $\ell$-restricted $\mathbb{F}_\ell$-morphism below:
\begin{equation*}
w_{\bar{v}}:\bar{\mathbf{E}}_\ell:=\mathrm{Res}_{\mathbb{E}_\ell/\mathbb{F}_\ell}(\mathbb{G}_m)\rightarrow \mathrm{GL}_{W_\ell}.
\end{equation*}
 Denote the image of $w_{\bar{v}}$ by $\bar{\mathbf{I}}_{\bar{v}}$ for $\bar{v}|\ell\gg1$. It is called the \textit{tame inertia torus at $\bar{v}\in\Sigma_{\bar{K}}$}.
\end{definition}

\begin{lemma}\label{244} There exists a constant $c_7=c_7(N)$ such that for any $\bar{v}|\ell>c_6(N)+1$, we have
\begin{enumerate}
\item[(i)] $\{\bar{\mathbf{I}}_{\bar{v}}\hookrightarrow \GL_{W_\ell}\}_{\bar{v}}$ have bounded formal characters (Definition \ref{bfc}');
\item[(ii)] $[\bar{\mathbf{I}}_{\bar{v}}(\mathbb{F}_\ell):f_{\bar{v}}(\mathbb{E}_\ell^*)]\leq c_7$.
\end{enumerate}
\end{lemma}

\begin{proof} Since dim$(W_\ell)$ and dim$(\bar{\mathbf{E}}_\ell)$ are bounded by a constant independent of $\ell$ and the exponents of the characters of $w_{\bar{v}}$ in terms of the fundamental characters \cite[$\mathsection3$]{Hall} belong to $[0,c_6]$, we find by 
Proposition \ref{p203} a set of characters $R_{\bar{v}}$ of uniformly bounded exponents of the diagonal subgroup of $\GL_{W_\ell}$ by diagonalizing $\bar{\mathbf{I}}_{\bar{v}}$ and then obtain assertion (i). For assertion (ii), uniform boundedness of exponents of characters and  $\mathrm{dim}(\bar{\mathbf{E}}_\ell)=c_4!$ (for all $\ell$) imply the number of connected components of $\mathrm{Ker}(w_{\bar{v}})$ is uniformly bounded by $c_7$. On the other hand, the number of $\mathbb{F}_\ell$-rational points of any $\mathbb{F}_\ell$-torus of dimension $k$ is between $(\ell-1)^k$ and $(\ell+1)^k$ by \cite[Lemma 3.5]{Nori}. Therefore, $\mu_\ell(I_{\bar{v}}^\mathrm{t})=f_{\bar{v}}(\mathbb{E}_\ell^*)$ has at least
\begin{equation*}
\frac{|\mathbb{E}_\ell^*|}{c_7(\ell+1)^{\mathrm{dim}(\mathrm{Ker}(w_{\bar{v}}))}}=\frac{\ell^{c_4!}-1}{c_7(\ell+1)^{\mathrm{dim}(\mathrm{Ker}(w_{\bar{v}}))}}
\end{equation*}
points and $[\bar{\mathbf{I}}_{\bar{v}}(\mathbb{F}_\ell):\mu_\ell(I_{\bar{v}}^\mathrm{t})]$ is bounded by 
\begin{equation*}
\frac{c_7(\ell+1)^{\mathrm{dim}(\mathrm{Ker}(w_{\bar{v}}))+\mathrm{dim}(\mathrm{Im}(w_{\bar{v}}))}}{\ell^{c_4!}-1}=\frac{c_7(\ell+1)^{c_4!}}{\ell^{c_4!}-1}\rightarrow c_7
\end{equation*}
when $\ell$ is big. This proves (ii).\end{proof}

\begin{lemma}\label{245}(Rigidity) \cite[$\mathsection3$]{Hall},\cite[$\mathsection3$]{SL2} Let $s\in\mathrm{GL}(W_\ell)$ be a semisimple element and $f_{\bar{v}}:\mathbb{E}_\ell^*\rightarrow \mathrm{GL}(W_\ell)$ a representation such that the exponents of characters of $f_{\bar{v}}$ belong to $[0,c]$ for some $c>0$. If $H\subset \mathbb{E}_\ell^*$ is a subgroup such that $f_{\bar{v}}(H)$ commutes with $s$ in $\mathrm{GL}(W_\ell)$ and $c\cdot [\mathbb{E}_\ell^*:H]\leq \ell-1$, then $\bar{\mathbf{I}}_{\bar{v}}$ commutes with $s$, and hence so does $f_{\bar{v}}(\mathbb{E}_\ell^*)$.\\
\end{lemma}

Recall from Definition \ref{comsys} that there is a finite subset $S\subset\Sigma_K$ such that $\phi_\ell$ is unramified outside
$S_\ell:=S\cup\{v\in\Sigma_K: v|\ell \}$ for all $\ell$.\\

\noindent\textit{\textbf{Proof of Theorem \ref{242}.}} The following arguments are influenced by the arguments Serre gave for \cite[Theorem 1]{SL2}. 

\begin{proof} Denote the image of $\mu_\ell(I_{\bar{v}}^{\mathrm{t}})$ under the map $\bar{\Gamma}_\ell/(\bar{\Gamma}_\ell\cap \bar{\mathbf{S}}_\ell(\mathbb{F}_\ell))\hookrightarrow\mathrm{GL}(W_\ell)$ by $\bar{\Omega}_{\bar{v}}$ whenever $\bar{v}|\ell$. Let $\bar{J}_\ell$ be a maximal abelian normal subgroup of $\bar{\Omega}_\ell:=\mu_\ell(\mathrm{Gal}_K)$.  We first prove that $\bar{\Omega}_{\bar{v}}$ commutes with $\bar{J}_\ell$ if $\ell$ is large. Since $\bar{\Omega}_{\bar{v}}$ and $\bar{J}_\ell$ are abelian and
\begin{equation*}
[\bar{\Omega}_{\bar{v}}:\bar{\Omega}_{\bar{v}}\cap \bar{J}_\ell]\leq c_5
\end{equation*}
by Theorem \ref{Jordan} (Jordan), the tame inertia torus $\bar{\mathbf{I}}_{\bar{v}}$ at $\bar{v}$ (Definition \ref{tori}) and hence  $f_{\bar{v}}(\mathbb{E}_\ell^*)=\bar{\Omega}_{\bar{v}}$  commute with $\bar{J}_\ell$ if $\ell>c_5c_6+1$ by rigidity (Lemma \ref{245}). For any $\bar{v}_1,\bar{v}_2|\ell$, since $\bar{\Omega}_{\bar{v}_1}\cap \bar{J}_\ell$ commutes with $\bar{\Omega}_{\bar{v}_2}\cap \bar{J}_\ell$ which are of bounded index in $\bar{\Omega}_{\bar{v}_1}$ and $\bar{\Omega}_{\bar{v}_2}$ respectively, we obtain $\bar{\mathbf{I}}_{\bar{v}_1}$ commutes with $\bar{\mathbf{I}}_{\bar{v}_2}$ if $\ell\gg1$ by rigidity (Lemma \ref{245}). The subgroup $\bar{H}_\ell$ of $\bar{\Omega}_\ell$ generated by the inertia subgroups $\bar{\Omega}_{\bar{v}}$ for all $\bar{v}|\ell$ is abelian and normal for $\ell\gg1$. As $\bar{J}_\ell$ is maximal normal abelian in $\bar{\Omega}_\ell$, $\bar{H}_\ell\subset \bar{J}_\ell$ for all $\ell\gg1$. Therefore,  $\bar{\Omega}_\ell/\bar{J}_\ell$ corresponds to a field extension of $K$ of degree bounded by $c_5$ that only ramifies in $S$ (Definition \ref{comsys}) for $\ell\gg1$. By Hermite's Theorem \cite[p.122]{Lang}, the composite of these fields is still a finite field extension $K'$ of $K$. Therefore, $\mu_\ell(\mathrm{Gal}_{K'})\subset \bar{J}_\ell$ for $\ell\gg1$.

 Since the representations $\{\phi_\ell\}$ come from \'etale cohomology and $I_{\bar{v}}\cap \mathrm{Gal}_{K'}$ is the inertia subgroup of $\mathrm{Gal}_{K'}$ at $\bar{v}$ \cite[Proposition 9.5]{Neu}, they are potentially semi-stable which means there exists a finite extension $K''$ of $K'$ such that $\phi_\ell(I_{\bar{v}}\cap \mathrm{Gal}_{K''})$ is unipotent for any $\bar{v}$ not dividing $\ell$ \cite[$\mathsection1$]{de}. Therefore, for each $\ell\gg1$ we have a finite abelian extension of $K''$ with Galois group $\mu_\ell(\mathrm{Gal}_{K''})$ contained in $\bar{J}_\ell$ that only ramifies at $v\in\Sigma_{K''}$ dividing $\ell$. Since $\mu_\ell(\Gal_{K''})$ is an abelian Galois group over $K''$, each ramified prime $v\in\Sigma_{K''}$ dividing large $\ell$  corresponds to an inertia subgroup $\bar{I}_v''\subset \mu_\ell(\mathrm{Gal}_{K''})$ and there are at most $[K'':\Q]$ of them. For each inertia subgroup $\bar{I}_v''$, choose a tame inertia torus $\bar{\mathbf{I}}_{\bar{v}}$ such that $\bar{I}_v''\subset  \bar{\mathbf{I}}_{\bar{v}}(\F_\ell)$. Since these tame inertia tori commute with each other, the algebraic group  $\bar{\mathbf{I}}_\ell$ generated by them is an $\F_\ell$-torus, called \emph{the inertia torus at $\ell$}. Since $\{\bar{\mathbf{I}}_{\bar{v}}\rightarrow \GL_{W_\ell}\}_{\bar{v}|\ell\gg1}$ have bounded formal characters (Lemma \ref{244}(i)) and each $\bar{\mathbf{I}}_\ell$ is generated by at most $[K'':\Q]$ tame inertia tori,   $\{\bar{\mathbf{I}}_\ell\hookrightarrow\GL_{W_\ell}\}_{\ell\gg1}$ have bounded formal characters by Proposition \ref{p204}. This proves (i). 
 
Let $\bar{I}_\ell''$ be the subgroup of $\mu_\ell(\mathrm{Gal}_{K''})$ generated by $\bar{I}_v''$ for all $v|\ell$. Then, for $\ell\gg1$ we have 
 \begin{equation*}
 \mu_\ell(\mathrm{Gal}_{K''})/\bar{I}_\ell''
 \end{equation*}
is the Galois group of a finite abelian extension of $K''$ that is unramified at every non-Archimedean valuation. By abelian class field theory, these fields generate a finite extension $K'''$ of $K''$. Choose $L$ normal over $K$ such that $K'''\subset L$. Then, we obtain
\begin{equation*}
(\ast):~~~\mu_\ell(\mathrm{Gal}_L)\subset \bar{I}_\ell''\subset \bar{\mathbf{I}}_\ell(\F_\ell).
\end{equation*}
It remains to prove (ii). Suppose $\bar{\mathbf{I}}_\ell$ is generated by tame inertia tori $\bar{\mathbf{I}}_{\bar{v}_i}$ for $1\leq i\leq k$ for some fixed $k\leq [K'':\Q]$. We have
\begin{equation*}
[\bar{\mathbf{I}}_\ell(\F_\ell):\mu_\ell(\mathrm{Gal}_L)]=[\bar{\mathbf{I}}_\ell(\F_\ell):\bar{\mathbf{I}}_\ell(\F_\ell)\cap\bar{\Omega}_\ell]\cdot[\bar{\mathbf{I}}_\ell(\F_\ell)\cap\bar{\Omega}_\ell:\mu_\ell(\mathrm{Gal}_L)]\end{equation*}
\begin{equation*}
\leq [\bar{\mathbf{I}}_\ell(\F_\ell): f_{\bar{v}_1}(\mathbb{E}_\ell^*)\cdots f_{\bar{v}_k}(\mathbb{E}_\ell^*)]\cdot[L:K].\vspace{.1in}
\end{equation*}
It suffices to show $[\bar{\mathbf{I}}_\ell(\F_\ell): f_{\bar{v}_1}(\mathbb{E}_\ell^*)\cdots f_{\bar{v}_k}(\mathbb{E}_\ell^*)]$ is bounded independent of $\ell$. The proof is identical to Lemma \ref{244}(ii) since $f_{\bar{v}_1}(\mathbb{E}_\ell^*)\cdots f_{\bar{v}_k}(\mathbb{E}_\ell^*)$ is the image of 
\begin{equation*}
f_{\bar{v}_1}\times\cdots\times f_{\bar{v}_k}: (\mathbb{E}_\ell^*)^k \to \mathrm{GL}(W_\ell),
\end{equation*}
$\bar{\mathbf{I}}_\ell$ is the image of
\begin{equation*}
w_{\bar{v}_1}\times\cdots\times w_{\bar{v}_k}: (\bar{\mathbf{E}}_\ell)^k \to \mathrm{GL}_{W_\ell},
\end{equation*}
$k$ (depending on $\ell$) is always less than $[K'':\Q]$,
and the exponents of characters ($\ell$-restricted \ref{tori}) of $w_{\bar{v}_1}\times\cdots\times w_{\bar{v}_k}$ are uniformly bounded. Therefore, there exists $c_8=c_8(N)$ such that $[\bar{\mathbf{I}}_\ell(\F_\ell):\mu_\ell(\mathrm{Gal}_L)]\leq c_8$ for $\ell\gg1$.
\end{proof}

\subsection{Construction of $\bar{\mathbf{G}}_\ell$} An $\mathbb{F}_\ell$-torus $\bar{\mathbf{I}}_\ell\subset\mathrm{GL}_{W_\ell}$ is constructed in $\mathsection 2.4$ for $\ell\gg1$ and we have the following map defined in Theorem \ref{241}
\begin{equation*}
t_\ell:\bar{\mathbf{N}}_\ell\twoheadrightarrow \bar{\mathbf{N}}_\ell/\bar{\mathbf{S}}_\ell \hookrightarrow\mathrm{GL}_{W_\ell}.
\end{equation*}
One has to show that $\bar{\mathbf{I}}_\ell\subset t_\ell(\bar{\mathbf{N}}_\ell)$ so that $t_\ell^{-1}(\bar{\mathbf{I}}_\ell)$ is connected. It suffices to consider tame inertia tori $\bar{\mathbf{I}}_{\bar{v}}$. Recall vector space $U_\ell$ from Theorem \ref{241}.

\begin{lemma}\label{251} Let $\bar{\mathbf{H}}_\ell$ be an algebraic subgroup of $\mathrm{GL}_{\bar{V}_\ell}$. Then $\bar{\mathbf{H}}_\ell$ acts on $\bar{U}_\ell$. If $\bar{\mathbf{H}}_\ell$ is invariant on the subspace 
\begin{equation*}
\bar{W}_\ell\subset \bar{U}_\ell
\end{equation*}
fixed by $\bar{\mathbf{S}}_\ell$, then $\bar{\mathbf{H}}_\ell$ is contained in $\bar{\mathbf{N}}_\ell$.
\end{lemma}

\begin{proof} Let $x\in \bar{\mathbf{H}}_\ell \backslash \bar{\mathbf{N}}_\ell$. Then there exists $s\in \bar{\mathbf{S}}_\ell$ such that $xsx^{-1}\notin \bar{\mathbf{S}}_\ell$. There exists $w\in \bar{W}_\ell$ such that 
\begin{equation*}
xsx^{-1}w\neq w
\end{equation*}
by the last statement of Theorem \ref{241}. Therefore,
\begin{equation*}
sx^{-1}w\neq x^{-1}w
\end{equation*}
implies $x^{-1}w\notin\bar{W}_\ell$, a contradiction. Hence, $\bar{\mathbf{H}}_\ell$ is contained in $\bar{\mathbf{N}}_\ell$.
\end{proof}

\begin{proposition}\label{252} The $\mathbb{F}_\ell$-torus $\bar{\mathbf{I}}_\ell$ in $\mathrm{GL}_{W_\ell}$ is a subgroup of the image of 
\begin{equation*}
t_\ell:\bar{\mathbf{N}}_\ell\twoheadrightarrow \bar{\mathbf{N}}_\ell/\bar{\mathbf{S}}_\ell \hookrightarrow\mathrm{GL}_{W_\ell}
\end{equation*}
defined in Theorem \ref{241}.
\end{proposition}

\begin{proof} Let $\bar{v}|\ell$ be a valuation of $\bar{K}$ and $I_{\bar{v}}$ the inertia subgroup of $\mathrm{Gal}_K$ at $\bar{v}$. The restriction $\phi_\ell: I_{\bar{v}}\rightarrow \mathrm{GL}(V_\ell)$ factors through a finite quotient $\pi_{\bar{v}}:I_{\bar{v}}\twoheadrightarrow J_{\bar{v}}$ such that $|J_{\bar{v}}|=\ell^k\cdot(\ell^{c_4!}-1)$. Recall vector spaces $W_\ell\subset U_\ell$ from Theorem \ref{241} and $f_{\bar{v}}:\mathbb{E}_\ell^*\rightarrow \GL(W_\ell)$ from Definition \ref{tori}.
Consider the following diagram so that 
$$r_\ell\circ \phi_\ell\circ i_{\bar{v}}=f_{\bar{v}}'$$
and the actions of $\mathbb{E}_\ell^*$ on $W_\ell$ via $f_{\bar{v}}'$ and $f_{\bar{v}}$ are the same.
Here $r_\ell$ is the obvious map and $i_{\bar{v}}$ is a splitting of $\pi_{\bar{v}}$. 
This is possible because $\mathbb{E}_\ell^*$ defined in  $\mathsection2.4$ is cyclic of order $(\ell^{c_4!}-1)$ prime to $\ell$.  
\begin{equation*}
\xymatrix{
J_{\bar{v}} \ar[d]_{\phi_\ell} \ar@{->>}[r]_{\pi_{\bar{v}}}  &\mathbb{E}_\ell^*\ar@/_1pc/[l]_{i_{\bar{v}}} \ar[d]^{f_{\bar{v}}'}\\
\mathrm{GL}_{V_\ell} \ar[r]^{r_\ell} &\mathrm{GL}_{U_\ell}}
\end{equation*}

If $\ell$ is sufficiently large, then the exponents of the characters ($\ell$-restricted) of representations $\phi_\ell\circ i_{\bar{v}}$ and $r_\ell\circ\phi_\ell\circ i_{\bar{v}}$ belong to $[0,i]$ and $[0,ic_3]$  respectively by Theorem \ref{233} and the construction of $U_\ell$. Recall $\bar{\mathbf{E}}_\ell$ from definition \ref{tori}. By Weil restriction of scalars, we obtain two $\mathbb{F}_\ell$-morphisms
\begin{equation*}
\alpha_\ell:\bar{\mathbf{E}}_\ell\to \mathrm{GL}_{V_\ell}
\end{equation*}
\begin{equation*}
\beta_\ell:\bar{\mathbf{E}}_\ell\to \mathrm{GL}_{U_\ell}.
\end{equation*} 
Since $r_\ell\circ\alpha_\ell$ and $\beta_\ell$ are both $\ell$-restricted \cite[$\mathsection3$]{Hall} and equal to $r_\ell\circ\phi_\ell\circ i_{\bar{v}}$ when restricting to $\mathbb{E}_\ell^*$, by uniqueness \cite[$\mathsection3$]{Hall} we have
\begin{equation*}
r_\ell\circ\alpha_\ell=\beta_\ell.
\end{equation*}

The image $r_\ell\circ\phi_\ell\circ i_{\bar{v}}(\mathbb{E}_\ell^*)=f_{\bar{v}}'(\mathbb{E}_\ell^*)$ maps $W_\ell$ and hence $\bar{W}_\ell$ to itself, so $\beta_\ell(\bar{\mathbf{E}}_\ell)$ also maps $\bar{W}_\ell$ to itself. Since $r_\ell\circ\alpha_\ell(\bar{\mathbf{E}}_\ell)=\beta_\ell(\bar{\mathbf{E}}_\ell)$, we conclude that $\alpha_\ell(\bar{\mathbf{E}}_\ell)\subset \bar{\mathbf{N}}_\ell$ by Lemma \ref{251}. One also observes that the following morphism
\begin{equation*}
t_\ell:\bar{\mathbf{N}}_\ell\twoheadrightarrow \bar{\mathbf{N}}_\ell/\bar{\mathbf{S}}_\ell \hookrightarrow\mathrm{GL}_{W_\ell}
\end{equation*}
maps $\alpha_\ell(\bar{\mathbf{E}}_\ell)$ to $\bar{\mathbf{I}}_{\bar{v}}:=w_{\bar{v}}(\bar{\mathbf{E}}_\ell)$. Therefore, tame inertia torus $\bar{\mathbf{I}}_{\bar{v}}$ and thus $\bar{\mathbf{I}}_\ell$ is a subgroup of $t_\ell(\bar{\mathbf{N}}_\ell)$. \end{proof} 
\vspace{.1in}

\begin{definition}\label{normal} Let $L$ be the normal extension of $K$ in Theorem \ref{242}. Denote $\phi_\ell(\mathrm{Gal}_L)$ by $\bar{\gamma}_\ell$ for all $\ell$. Then $[\bar{\Gamma}_\ell:\bar{\gamma}_\ell]\leq [L:K]$ for all $\ell$.
\end{definition}

\noindent\textit{\textbf{Proof of Theorem \ref{redenv}(i),(ii).}}

\begin{proof}Since $\bar{\mathbf{S}}_\ell$ is a connected normal subgroup of $\bar{\mathbf{N}}_\ell$, $\bar{\mathbf{I}}_\ell$ is a torus, and $t_\ell$ is an $\mathbb{F}_\ell$-morphism, Proposition \ref{252} implies $t_\ell^{-1}(\bar{\mathbf{I}}_\ell)$, the preimage of the $\mathbb{F}_\ell$-torus $\bar{\mathbf{I}}_\ell$ is a connected $\mathbb{F}_\ell$-reductive group $\bar{\mathbf{G}}_\ell$. 
Moreover, $\bar{\gamma}_\ell\subset \bar{\mathbf{G}}_\ell(\F_\ell)$ by construction of $\bar{\mathbf{G}}_\ell$ for $\ell\gg1$. We obtain an exact sequences of $\F_\ell$ algebraic groups for $\ell\gg1$
\begin{equation*}
1\rightarrow \bar{\mathbf{S}}_\ell\rightarrow \bar{\mathbf{G}}_\ell\rightarrow \bar{\mathbf{I}}_\ell\rightarrow 1.
\end{equation*}
and hence 
\begin{equation*}
1\rightarrow \bar{\mathbf{S}}_\ell(\mathbb{F}_\ell)\rightarrow \bar{\mathbf{G}}_\ell(\mathbb{F}_\ell)\rightarrow \bar{\mathbf{I}}_\ell(\mathbb{F}_\ell).
\end{equation*}
Recall $\mu_\ell(\mathrm{Gal}_L)=t_\ell(\bar{\gamma}_\ell)$ from Theorem \ref{242}. Since the semisimple envelopes (Definition \ref{Nori}) of $\bar{\Gamma}_\ell$ and $\bar{\gamma}_\ell$ are identical for $\ell\gg1$ by Remark \ref{ind}, the above exact sequence implies 
\begin{equation*}
[\bar{\mathbf{G}}_\ell(\mathbb{F}_\ell):\bar{\gamma}_\ell]\leq [\bar{\mathbf{S}}_\ell(\mathbb{F}_\ell):\bar{\gamma}_\ell\cap \bar{\mathbf{S}}_\ell(\mathbb{F}_\ell)][\bar{\mathbf{I}}_\ell(\mathbb{F}_\ell):\mu_\ell(\mathrm{Gal}_L)]\leq 2^{N-1}c_8
\end{equation*}
by Proposition \ref{N2}(iii) and Theorem \ref{242} for $\ell\gg1$. Since the derived group of $\bar{\mathbf{G}}_\ell$ is $\bar{\mathbf{S}}_\ell$, the action of $\bar{\mathbf{G}}_\ell$ on the ambient space is semisimple if $\ell\gg1$ by Proposition \ref{N2}(ii). Therefore, we have proved Theorem \ref{redenv} (i) and (ii).\end{proof}

\noindent\textit{\textbf{Proof of Theorem \ref{redenv}(iii).}}

\begin{proof}Let $\bar{\mathbf{S}}_\ell^{\sc}\rightarrow \bar{\mathbf{S}}_\ell$ be the simply connected cover of $\bar{\mathbf{S}}_\ell$. The representation
$(\bar{\mathbf{S}}^{\sc}_\ell\rightarrow \bar{\mathbf{S}}_\ell\hookrightarrow\mathrm{GL}_{N,\mathbb{F}_\ell})\times\bar\F_\ell$
 is semisimple and has a $\mathbb{Z}$-form which belongs to a finite set of $\mathbb{Z}$-representations of simply-connected Chevalley schemes \cite[Theorem 24]{EHK} if $\ell\gg1$. Thus, $\{\bar{\mathbf{S}}_\ell\hookrightarrow\mathrm{GL}_{N,\mathbb{F}_\ell}\}_{\ell\gg1}$ have bounded formal characters (Definition \ref{bfc}'). Let $\bar{\mathbf{C}}_\ell$ be the center of $\bar{\mathbf{G}}_\ell$. Since $\bar{\mathbf{S}}_\ell$ acts semi-simply on $\bar{V}_\ell$ by Proposition \ref{N2}(ii) for $\ell\gg1$, we decompose the representation $\bar{\mathbf{S}}_\ell\rightarrow\mathrm{GL}(\bar{V}_\ell)$
into a sum of absolutely irreducible representations $\bar{M}_i$
\begin{equation*}
\bar{V}_\ell= (\bigoplus_1^{m_1}\bar{M}_1)\oplus(\bigoplus_1^{m_2}\bar{M}_2)\oplus\cdots \oplus (\bigoplus_1^{m_k}\bar{M}_k)
 \end{equation*}
  such that $\bar{M}_i\ncong \bar{M}_j$ if $i\neq j$. If $c\in\bar{\mathbf{C}}_\ell$, then $\bar{M}_i$ and $c(\bar{M}_i)$ are isomorphic representations of $\bar{\mathbf{S}}_\ell$ for all $i$. Hence, $c$ is invariant on $\oplus_1^{m_i}\bar{M}_i$ and $\oplus_1^{m_i}\bar{M}_i$ is a subrepresentation of $\bar{\mathbf{G}}_\ell$ on $\bar{V}_\ell$ for all $i$. Let $n_i$ be the dimension of $\bar{M}_i$. Denote the representation of $\bar{\mathbf{S}}_\ell$ on $\bar{M}_i$ under some coordinates by 
 \begin{equation*}
 u_i:\bar{\mathbf{S}}_\ell\rightarrow \mathrm{GL}_{n_i}(\bar{\mathbb{F}}_\ell).
 \end{equation*}
Then, the representation of $\bar{\mathbf{G}}_\ell$ on $\bigoplus_1^{m_i}\bar{M}_i$ is given by:
\begin{equation*}
q_i:\bar{\mathbf{G}}_\ell\rightarrow\mathrm{GL}_{n_im_i}(\bar{\mathbb{F}}_\ell)
\end{equation*}
so that when restricting to $\bar{\mathbf{S}}_\ell$, the action is ``diagonal''
 \begin{equation*}
 q_i:\bar{\mathbf{S}}_\ell\stackrel{u_i}{\rightarrow}\mathrm{GL}_{n_i}(\bar{\mathbb{F}}_\ell)\rightarrow \bigoplus_1^{m_i} \mathrm{GL}_{n_i}(\bar{\mathbb{F}}_\ell)\subset\mathrm{GL}_{n_im_i}(\bar{\mathbb{F}}_\ell)
 \end{equation*}
 \begin{equation*}
x\mapsto u_i(x)\mapsto (u_i(x),...,u_i(x)).\hspace{.5in} 
 \end{equation*}
Since $u_i$ is a irreducible representation and $q_i(c)$ commutes with $q_i(\bar{\mathbf{S}}_\ell)$, $q_i(c)$ is contained in the subgroup
 \begin{equation*}
  \bar{\mathbf{H}}_i= \begin{pmatrix}
      \bar{\mathbf{D}}_{11} &\bar{\mathbf{D}}_{12} &... &\bar{\mathbf{D}}_{1m_i}\\
      \bar{\mathbf{D}}_{21} &\bar{\mathbf{D}}_{22}  &... &\bar{\mathbf{D}}_{2m_i}\\
      \vdots &\vdots &\ddots &\vdots\\
      \bar{\mathbf{D}}_{m_i1} &\bar{\mathbf{D}}_{m_i2} &... &\bar{\mathbf{D}}_{m_im_i}
\end{pmatrix},
\end{equation*}
where $\bar{\mathbf{D}}_{jk}$ is the subgroup of scalars of $\mathrm{GL}_{n_i}(\bar{\mathbb{F}}_\ell)$ for all $1\leq j\leq m_i$, $1\leq k\leq m_i$. We see that $\bar{\mathbf{H}}_i$ is isomorphic to $\mathrm{GL}_{m_i}(\bar{\mathbb{F}}_\ell)$. Since $q_i(\bar{\mathbf{C}}_\ell)$ is a diagonalizable group which commutes with $q_i(\bar{\mathbf{S}}_\ell)$ and $q_i|_{\bar{\mathbf{S}}_\ell}$ is ``diagonal'', we may assume $q_i(\bar{\mathbf{C}}_\ell)$ is contained in the following torus $\bar{\mathbf{D}}_i$ for all $i$
 \begin{equation*}
   \bar{\mathbf{D}}_i=\begin{pmatrix}
      \bar{\mathbf{D}}_{11} &0 &... &0\\
      0 &\bar{\mathbf{D}}_{22}  &... &0\\
      \vdots &\vdots &\ddots &\vdots\\
      0 &0 &... &\bar{\mathbf{D}}_{m_im_i}
\end{pmatrix}
\end{equation*}
after a change of coordinates by some element in $\bar{\mathbf{H}}_i\cong \mathrm{GL}_{m_i}(\bar{\mathbb{F}}_\ell)$. Therefore, we may assume that $\bar{\mathbf{C}}_\ell$ is a subgroup of
\begin{equation*}
\bar{\mathbf{B}}_\ell:=\bar{\mathbf{D}}_1\times \bar{\mathbf{D}}_2\times\cdots\times \bar{\mathbf{D}}_k\subset\mathrm{GL}_N(\bar{\mathbb{F}}_\ell).
\end{equation*}
in suitable coordinates. Torus $\bar{\mathbf{B}}_\ell$ centralizes $\bar{\mathbf{S}}_\ell$ implies $\bar{\mathbf{B}}_\ell\subset \bar{\mathbf{N}}_\ell$. Denote the restriction $t_\ell|_{\bar{\mathbf{B}}_\ell}$ by $s_\ell$. Since $\bar{\mathbf{N}}_\ell$ acts on $\bar{W}_\ell$, we have 
\begin{equation*}
s_\ell:\bar{\mathbf{B}}_\ell\rightarrow \mathrm{GL}_{W_\ell}.
\end{equation*}
We obtain $(s_\ell^{-1}(\bar{\mathbf{I}}_\ell))^\circ=\bar{\mathbf{C}}_\ell^\circ$ because $\mathrm{Ker}(s_\ell)$ is discrete. Consider the construction of $U_\ell$ from Theorem \ref{241}. This implies the exponents of characters of $s_\ell$ on $\bar{\mathbf{D}}_i\cong\prod_1^{m_i}\bar{\mathbb{F}}_\ell^*$ are between $0$ and $c_3$ for all $i$. By Theorem \ref{242}(i) and above, the diagonalizable groups $\{s_\ell^{-1}(\bar{\mathbf{I}}_\ell)\}_{\ell\gg1}$ satisfies the bounded exponents condition in Definition \ref{bfc}'. Hence, $\{\bar{\mathbf{C}}_\ell^\circ=(s_\ell^{-1}(\bar{\mathbf{I}}_\ell))^\circ\hookrightarrow \bar{\mathbf{B}}_\ell\hookrightarrow\GL_{V_\ell}\}_{\ell\gg1}$ have bounded formal characters. Since $\{\bar{\mathbf{C}}_\ell\hookrightarrow\GL_{N,\F_\ell}\}_{\ell\gg1}$ and $\{\bar{\mathbf{S}}_\ell\hookrightarrow\GL_{N,\F_\ell}\}_{\ell\gg1}$ both have bounded formal characters and $\bar{\mathbf{C}}_\ell^\circ$ commutes with $\bar{\mathbf{S}}_\ell$ for $\ell\gg1$, $\{\bar{\mathbf{G}}_\ell=\bar{\mathbf{C}}_\ell^\circ\cdot\bar{\mathbf{S}}_\ell\hookrightarrow\GL_{N,\F_\ell}\}_{\ell\gg1}$ have bounded formal characters by Proposition \ref{p204}. This prove Theorem \ref{redenv}(iii).\end{proof}

\section{$\ell$-independence of $\bar{\Gamma}_\ell$}
 
\subsection{Formal character of $\bar{\mathbf{G}}_\ell\subset\GL_{N,\F_\ell}$}
A system of algebraic envelopes $\{\bar{\mathbf{G}}_\ell\}_{\ell\gg1}$ of $\{\bar{\Gamma}_\ell\}_{\ell\gg1}$ (Definition \ref{env}) are constructed in $\mathsection2.5$. Let $\mathbf{G}_\ell$ be the algebraic monodromy group of $\Phi_\ell^{\ss}$ for all $\ell$. The compatibility (Definition \ref{comsys}) of the system $\{\phi_\ell\}$ implies that the formal characters of $\{\bar{\mathbf{G}}_\ell\hookrightarrow \GL_{N,\F_\ell}\}_{\ell\gg1}\cup \{\mathbf{G}_\ell\hookrightarrow \GL_{N,\Q_\ell}\}_{\ell\gg1}$ are the same in the sense of Definition \ref{fc}':

\begin{theorem}\label{311} Let $\{\bar{\mathbf{G}}_\ell\}_{\ell\gg1}$ be a system of algebraic envelopes of $\{\bar{\Gamma}_\ell\}_{\ell\gg1}$ (Definition \ref{env}).
\begin{enumerate}
\item [(i)] The formal characters of $\bar{\mathbf{G}}_\ell\hookrightarrow \mathrm{GL}_{N,\mathbb{F}_\ell}$ and $\mathbf{G}_\ell\hookrightarrow \GL_{N,\Q_\ell}$ are the same for $\ell\gg1$.
\item [(ii)] The formal characters of $\{\bar{\mathbf{G}}_\ell\hookrightarrow \mathrm{GL}_{N,\mathbb{F}_\ell}\}_{\ell\gg1}$ are the same.
\end{enumerate}\end{theorem}

\begin{proof}  The mod $\ell$ system $\{\phi_\ell:\mathrm{Gal}_K\rightarrow \mathrm{GL}_N(\mathbb{F}_\ell)\}$ comes from the $\ell$-adic system $\{ \Phi_\ell^{\ss}:\mathrm{Gal}_K\rightarrow \mathrm{GL}_N(\Q_\ell)\}$ (Definition \ref{arise}). The algebraic monodromy group $\mathbf{G}_\ell$ is reductive for all $\ell$. By taking a finite extension $K^{\mathrm{conn}}$ of $K$ \cite{SL1}, we may assume $\mathbf{G}_\ell$  is connected  for all $\ell$. This does not change the formal character of $\mathbf{G}_\ell\hookrightarrow\GL_{N,\Q_\ell}$. It is well known that these algebraic monodromy groups have same reductive rank $r$.  Define
\begin{equation*}
\textit{Char}:\mathrm{GL}_N\to \mathbb{G}_a^{N-1}\times\mathbb{G}_m
\end{equation*}
that maps a matrix to the coefficients of its characteristic polynomial. We know that $\textit{Char}(\mathbf{G}_\ell)$  is a $\mathbb{Q}$-variety of dimension $r$ that is independent of $\ell$ (by the compatibility conditions) and can be defined over $\Z[\frac{1}{N'}]$ for some positive integer $N'$ that is sufficiently divisible. Let  $\mathbf{P}_{\mathbb{Z}[\frac{1}{N'}]}$ be the Zariski closure of $\textit{Char}(\mathbf{G}_\ell)$ in  the projective $\mathbb{P}^N_{\mathbb{Z}[\frac{1}{N'}]}$. Since $\phi_\ell$ is continuous, every element of $\bar{\Gamma}_\ell$ is the image of a Frobenius element. Therefore, $\textit{Char}(\bar{\Gamma}_\ell)$ is a subset of the $\mathbb{F}_\ell$-rational points of $\mathbf{P}_{\mathbb{F}_\ell}:=\mathbf{P}_{\mathbb{Z}[\frac{1}{N'}]}\times_\Z \F_\ell$ for $\ell\gg1$.

 Generic flatness \cite[Theorem 6.9.1]{EGA} implies $\mathbf{P}_{\mathbb{Z}[\frac{1}{N'}]}$ is flat over $\mathbb{Z}[\frac{1}{N'}]$ for sufficiently divisible $N'$, so the dimension of every irreducible component of $\mathbf{P}_{\mathbb{Z}[\frac{1}{N'}]}$ is $r+1$ \cite[Chapter 3 Proposition 9.5]{Ha} and hence the dimension of every irreducible component of $\mathbf{P}_{\mathbb{F}_\ell}$ is $r$ \cite[Chapter 3 Corollary 9.6]{Ha} for $\ell\gg1$. Also, the Hilbert polynomial of $\mathbf{P}_{\mathbb{F}_\ell}$ and  in particular the \emph{degree} (let it be $d$) of $\mathbf{P}_{\mathbb{F}_\ell}\subset\mathbb{P}^N_{\mathbb{F}_\ell}$ is independent of $\ell$ for $\ell\gg1$ \cite[Chapter 3 Theorem 9.9]{Ha}.  Since $d$ is a positive integer, we conclude that the number and degrees of irreducible components of $\mathbf{P}_{\mathbb{F}_\ell}$ are bounded by $d$ \cite[Chapter 1 Proposition 7.6(a),(b)]{Ha}. By \cite[Theorem 1]{LW} and above, we have
 \begin{equation*}
 |\mathbf{P}_{\mathbb{F}_\ell}(\mathbb{F}_\ell)|\leq 3d\cdot\ell^r
 \end{equation*}
for $\ell\gg1$. Let $\bar{\mathbf{T}}_\ell$ be a $\mathbb{F}_\ell$-maximal torus of $\bar{\mathbf{G}}_\ell$. \cite[Lemma 3.5]{Nori} implies $\bar{\mathbf{T}}_\ell$ has at least $(\ell-1)^{\mathrm{dim}(\bar{\mathbf{T}}_\ell)}$ $\F_\ell$-rational points. By Theorem \ref{redenv} (i), there is an integer $n>0$ such that the $n$th power of $\bar{\mathbf{T}}_\ell(\mathbb{F}_\ell)$ is contained in $\bar{\gamma}_\ell$ for $\ell\gg1$. One sees by diagonalizing $\bar{\mathbf{T}}_\ell$ in $\GL_{N,\bar{F}_\ell}$ that the order of the kernel of this $n$th power homomorphism is less than or equal to $n^N$. Hence, we obtain
\begin{equation*}
|\bar{\mathbf{T}}_\ell(\mathbb{F}_\ell)\cap\bar{\gamma}_\ell|\geq \frac{(\ell-1)^{\mathrm{dim}(\bar{\mathbf{T}}_\ell)}}{n^N}.
\end{equation*}
Also, morphism $\textit{Char}$ restricted to any maximal torus of $\mathrm{GL}_N$  is finite morphism of degree $N!$. Therefore, there is a constant $c>0$ such that 
\begin{equation*}
c\cdot\ell^{\mathrm{dim}(\bar{\mathbf{T}}_\ell)}\leq |\textit{Char}(\bar{\mathbf{T}}_\ell(\mathbb{F}_\ell)\cap\bar{\gamma}_\ell)|\leq |\textit{Char}(\bar{\gamma}_\ell)|\leq  |\mathbf{P}_{\mathbb{F}_\ell}(\mathbb{F}_\ell)|\leq 3d\cdot\ell^r
\end{equation*}
for $\ell\gg1$. This implies $\mathrm{dim}(\bar{\mathbf{T}}_\ell)\leq r$  for  $\ell\gg1$.

On the other hand, we find for each $\ell\gg1$ a set  $R_\ell$ of characters of $\mathbb{G}_m^N$ of exponents bounded by $C>0$ such that $\bar{\mathbf{T}}_\ell$ is conjugate  in $\GL_{N,\bar{F}_\ell}$ to the kernel of $R_\ell$ by Theorem \ref{redenv}(iii) and Definition \ref{bfc}'.
Let $\mathscr{L}$ be an infinite subset of prime numbers $\mathscr{P}$ such that for all $\ell,\ell'\in\mathscr{L}$, we have equality $R_\ell=R_{\ell'}$. Denote this common set of characters by $R$ and define $\mathbf{Y}_\C=\{y\in\mathbb{G}_{m,\mathbb{C}}^N:\chi(y)=1\hspace{.1in}\forall \chi\in R\}$ so that $\mathrm{dim}_\C\mathbf{Y}_\C=\mathrm{dim}_{\bar{\F}_\ell}\bar{\mathbf{T}}_\ell$ for all $\ell\in\mathscr{L}$. If $\bar{v}$ divides $v\in \Sigma_K\backslash S_\ell$ ($S_\ell$ in Definition \ref{comsys}), then the characteristic polynomial of $\phi_\ell(\mathrm{Frob}_{\bar{v}})$ is just the  mod $\ell$ reduction of the characteristic polynomial of $\Phi_\ell^{\ss}(\mathrm{Frob}_{\bar{v}})=P_v(x)\in\mathbb{Q}[x]$ which depends only on $v$ (Definition \ref{comsys}). Therefore, for each $v\notin S$ (Definition \ref{comsys}), we can put the roots of $P_v(x)$ in some order $\alpha_1,\alpha_2,...,\alpha_N$  such that the following  congruence equation holds:
\begin{equation*}
\alpha_1^{m_1}\alpha_2^{m_2}\cdots\alpha_N^{m_N}\equiv 1\hspace{0.5in}(\mathrm{mod}\hspace{.1in}\ell')
\end{equation*}
for any character $x_1^{m_1}x_2^{m_2}\cdots x_N^{m_N}\in R$ and 
$$\ell'\in \mathscr{L}_v:=\mathscr{L}\backslash \{\ell''\in\mathscr{P}: \exists v'\in S_\ell~\mathrm{s.t.}~v'|\ell''\}$$
 if $v|\ell$. Since $\alpha_1^{m_1}\alpha_2^{m_2}\cdots\alpha_N^{m_N}$ is an algebraic number and $\mathscr{L}_v$ consists of infinitely many primes, we obtain equality
\begin{equation*}
\alpha_1^{m_1}\alpha_2^{m_2}\cdots\alpha_N^{m_N}=1
\end{equation*}
for any character $x_1^{m_1}x_2^{m_2}\cdots x_N^{m_N}\in R$. Therefore, 
\begin{equation*}
(\textit{Char}|_{\mathbb{G}_m^N})^{-1}(\{P_v(x): v\in\Sigma_K\backslash S\})\subset\bigcup_{g\in \mathrm{Perm}(N)} g(\mathbf{Y}_\C),
\end{equation*}
where  $\mathrm{Perm}(N)$ is the group of permutations of $N$ letters permuting the coordinates. Since $\{P_v(x):v\in\Sigma_K\backslash S\}$ is Zariski dense in $\textit{Char}(\mathbf{G}_\ell)$ of dimension $r$ and $\textit{Char}|_{\mathbb{G}_m^N}$ is a finite morphism of degree $N!$, the Zariski closure of $(\textit{Char}|_{\mathbb{G}_m^N})^{-1}(\{P_v(x):v\in\Sigma_K\backslash S\})$ in $\mathbb{G}_{m,\mathbb{C}}^N$ denoted by $\mathbf{D}_\C$ is also of dimension $r$. Since we have obtained $\mathrm{dim}(\bar{\mathbf{T}}_\ell)\leq r$ at the end of the second paragraph and any maximal torus of the algebraic monodromy group $\mathbf{G}_\ell$ is conjugate in $\GL_{N,\C}$ to an irreducible component of $\mathbf{D}_\C$ \cite{SL1}, the inclusion 
$$\mathbf{D}_\C\subset \bigcup_{g\in \mathrm{Perm}(N)} g(\mathbf{Y}_\C) $$
implies the formal characters of $\bar{\mathbf{G}}_\ell\hookrightarrow\GL_{N,\F_\ell}$ and $\mathbf{G}_\ell\hookrightarrow\GL_{N,\Q_\ell}$ are the same in the sense of Definition \ref{fc}' for all $\ell\in\mathscr{L}$. There are only finitely many possibilities for $R_\ell$ by 
Remark \ref{finitefc} and Proposition \ref{p203}. By excluding the primes $\ell$ such that $R_\ell$ appears finitely many times, we conclude that the formal characters of $\bar{\mathbf{G}}_\ell\hookrightarrow\GL_{N,\F_\ell}$ and $\mathbf{G}_\ell\hookrightarrow\GL_{N,\Q_\ell}$ are the same for $\ell\gg1$. This proves (i) and hence (ii) since formal character of $\mathbf{G}_\ell\hookrightarrow \GL_{N,\Q_\ell}$ is independent of $\ell$ \cite{SL1}.
\end{proof}

\subsection{Formal character of $\bar{\mathbf{S}}_\ell\subset\GL_{N,\F_\ell}$} We make the following assumptions for this subsection.

\noindent\textbf{Assumptions}: By taking a field extension of $K$, we may assume 
\begin{enumerate}
\item[(i)] $\mathbf{G}_\ell$, the algebraic monodromy group of $\Phi_\ell^{\ss}$ is connected for all $\ell$ (see \cite{SL1}),
\item[(ii)] $\bar{\Omega}_\ell:=\mu_\ell(\bar{\Gamma}_\ell)$ corresponds to an abelian extension of $K$ that is unramified at all primes not dividing $\ell$ for all $\ell$ (see the first paragraph of the proof of Theorem \ref{242}).\\
\end{enumerate}

Theorem \ref{321} below is the main result in this subsection. Denote a finite extension of $K$ by $K'$. Since $\bar{\mathbf{S}}_\ell$ is independent of $K'$ over $K$ for $\ell\gg1$ by Remark \ref{ind}, the assumptions above remain valid for $K'$, and $\{\bar{\mathbf{G}}_\ell\}_{\ell\gg1}$ constructed in $\mathsection2.5$ are still algebraic envelopes of $\{\phi_\ell(\mathrm{Gal}_{K'})\}_{\ell\gg1}$, we are free to replace $K$ by $K'$ in this subsection.

\begin{theorem}\label{321} Let $\bar{\mathbf{S}}_\ell\subset\GL_{N,\F_\ell}$ be the semisimple envelope of $\bar{\Gamma}_\ell$ (Definition \ref{Nori})  for all $\ell\gg1$. 
\begin{enumerate}
\item[(i)] The formal character of $\bar{\mathbf{S}}_\ell\hookrightarrow\mathrm{GL}_{N,\mathbb F_\ell}$ is equal to the formal character of $\mathbf{G}_\ell^\mathrm{der}\hookrightarrow\mathrm{GL}_{N,\Q_\ell}$ for $\ell\gg1$, where $\mathbf{G}_\ell^\mathrm{der}$ is the derived group of the algebraic monodromy group $\mathbf{G}_\ell$ of $\Phi_\ell^{\ss}$.
\item[(ii)] The formal character of $\bar{\mathbf{S}}_\ell\hookrightarrow \mathrm{GL}_{N,\mathbb{F}_\ell}$ is independent of $\ell$   if $\ell\gg1$.
\end{enumerate}\end{theorem}

In \cite[$\mathsection 3$]{Hui}, we used mainly abelian $\ell$-adic representations to prove that the formal character of $\mathbf{G}_\ell^\mathrm{der}\hookrightarrow\GL_{N,\Q_\ell}$ is independent of $\ell$. To prove Theorem \ref{321}, we adopt this strategy in a mod $\ell$ fashion. The key point is to prove that the inertia characters of $\mu_\ell$ (Definition \ref{mu}) for $\ell\gg1$ are in some sense the  mod $\ell$ reduction of inertia characters of some Serre group $\mathbf{S}_{\mathfrak{m}}$ \cite[Chapter 2]{Sbk} (Proposition \ref{324}).

\begin{definition}\label{Serre}
For each prime $\ell\in\mathscr{P}$, choose a valuation $\bar{v}_\ell$ of $\bar{\Q}$ that extends the $\ell$-adic valuation of $\Q$. 
This valuation on $\bar\Q$ is equal to the restriction of 
the unique non-Archimedean valuation on $\bar\Q_\ell$ (extending the $\ell$-adic valuation on $\Q_\ell$) 
to $\bar\Q$ with respect to some embedding $\bar\Q\hookrightarrow\bar\Q_\ell$.
Denote also this valuation on $\bar\Q_\ell$ by $\bar{v}_\ell$. 
Define the following notation.
\begin{itemize}
\item[\textbullet] $\mathrm{Gal}_K^{\mathrm{ab}}$: the Galois group of the maximal abelian extension of $K$,
\item[\textbullet] $I_K$: the group of id\'eles of $K$,
\item[\textbullet] $(x_v)_{v\in\Sigma_K}$: a representation of a finite id\'ele,
\item[\textbullet] $K_v$: the completion of $K$ with respect to $v\in\Sigma_K$,
\item[\textbullet] $U_v$:  the unit group of $K_v^*$, 
\item[\textbullet] $k_v$: the residue field of $K_v$,
\item[\textbullet] $\mathfrak{m}_0$: the modulus of empty support,
\item[\textbullet] $U_{\mathfrak{m}_0}:=\prod_v U_v$,
\item[\textbullet] $K_\ell:=\prod_{v|\ell}K_v=K\otimes \Q_\ell$,
\item[\textbullet] $\bar{\Z}_\ell$: the valuation ring of $\bar{v}_\ell$,
\item[\textbullet] $\mathfrak{p}_\ell$: the maximal ideal of $\bar{v}_\ell$,
\item[\textbullet] $k_\ell$: the residue field of $\bar{v}_\ell$,
\item[\textbullet] $x_\ell:=(x_v)_{v|\ell}$.
\end{itemize}
Let $\sigma:K\rightarrow \bar{\Q}$ be an embedding of $K$ in $\bar{\Q}$. The composition of $\sigma$ with $\bar{\Q}\hookrightarrow \bar{\Q}_\ell$ extends to a $\Q_\ell$-algebra homomorphism $\sigma_\ell:K_\ell\rightarrow \bar{\Q}_\ell$. 
\end{definition}

\begin{remark}\label{sigma} The field $k_\ell$ is an algebraic closure of $\F_\ell$ and homomorphism $\sigma_\ell$ is trivial on the components $K_v$ of $K_\ell$ when $v$ is not equivalent to $\bar{v}_\ell\circ\sigma$.\\\end{remark}

Recall representation $\mu_\ell:\mathrm{Gal}_K\rightarrow \mathrm{GL}(W_\ell)$ (abelian by Assumption (ii)) from Definition \ref{mu}.  Thus, $\mu_\ell$ induces $\rho_\ell$ below for each $\ell$ by composing with $I_K\rightarrow \mathrm{Gal}_K^{\mathrm{ab}}$:
\begin{equation*}
\rho_\ell: I_K\rightarrow \mathrm{GL}(W_\ell).
\end{equation*}

\begin{proposition}\label{323} If $\chi_\ell: I_K\rightarrow \bar{\F}_\ell^*$ is a character of $\rho_\ell$ for $\ell\gg1$, then for all finite id\'ele $x\in U_{\mathfrak{m}_0}$  we have the congruence 
\begin{equation*}
\chi_\ell(x)\equiv \prod_{\sigma\in\mathrm{Hom}(K,\bar{\mathbb{Q}})}\sigma_\ell(x_\ell^{-1})^{m(\sigma,\ell)}\hspace{.3in}(\mathrm{mod}\hspace{.05in}\mathfrak{p}_\ell)
\end{equation*}
such that $0\leq m(\sigma,\ell)\leq c_6$.
\end{proposition}

\begin{proof}
Since $|\bar{\Omega}_\ell|$ is prime to $\ell$, the following homomorphism 
\begin{equation*}
U_v\hookrightarrow K_v^*\rightarrow I_K\stackrel{\rho_\ell}{\rightarrow} \mathrm{GL}(W_\ell)
\end{equation*}
factors through $\alpha_v: k_v^*\rightarrow \mathrm{GL}(W_\ell)$ for all $v|\ell$. On the other hand, let $\bar{v}\in\Sigma_{\bar{K}}$ divide $\ell$.  Since $\bar{\Omega}_\ell$ is abelian and of order prime to $\ell$, the restriction of $\mu_\ell:\mathrm{Gal}_K\rightarrow \GL(W_\ell)$ to $I_{\bar{v}}$ factors through 
\begin{equation*}
I_{\bar{v}}\rightarrow I_{\bar{v}}^{\mathrm{t}}\stackrel{\cong}{\rightarrow} \varprojlim \mathbb{F}_{\ell^k}^*\rightarrow k_v^*
\end{equation*}
and induces $\beta_v:  k_v^*\rightarrow \mathrm{GL}(W_\ell)$ that depends on $v=\bar{v}|_{\bar{K}}$. By \cite[Proposition 3]{S72}, $\alpha_v$ and $\beta_v$ are inverse of each other. Since $f_{\bar{v}}$ (Definition \ref{tori}) factors through $\beta_v$ and the exponents of any character of $f_{\bar{v}}$ when expressed as a $\ell$-restricted (Definition \ref{restrict}) product of fundamental characters of level $c_4!$ are bounded by $c_6$ for $\ell\gg1$ ($\mathsection2.4$), the exponents of $\chi_\ell$ when expressed as a $\ell$-restricted product of fundamental characters of level $[k_v:\F_\ell]$ are also bounded by $c_6$ for $\ell\gg1$. Since $\rho_\ell$ is unramified at all $v$ not dividing $\ell$, $\rho_\ell$ is trivial on subgroup $\prod_{v\nmid\ell}U_v$
of $ U_{\mathfrak{m}_0}:=\prod_{v}U_v$. Therefore, we conclude the congruence for $\ell\gg1$.\end{proof}

\begin{definition}\label{theta} Let $\mathbf{S}_\mathfrak{m}$ be the Serre group of $K$  with modulus $\mathfrak{m}$ \cite[Chapter 2]{Sbk} and $\Theta: \mathbf{S}_\mathfrak{m}\rightarrow \mathbb{G}_{m,\bar{\Q}_\ell}$ a character of $\mathbf{S}_\mathfrak{m}$ over $\bar{\Q}_\ell$. Since the image of the abelian representation $\Theta_\ell$ attached to $\Theta$ \cite[Chapter 2]{Sbk} 
\begin{equation*}
\Theta_\ell: \mathrm{Gal}_K^{\mathrm{ab}}\rightarrow \mathbf{S}_{\mathfrak{m}}(\Q_\ell)\stackrel{\Theta}{\rightarrow} \bar\Q_\ell^*
\end{equation*}
is contained in $\bar{\Z}_\ell^*$, define 
\begin{equation*}
\theta_\ell: I_K\rightarrow k_\ell^*\cong\bar{\F}_\ell^*
\end{equation*}
as the mod $\mathfrak{p}_\ell$ reduction of the composition of $I_K\rightarrow \mathrm{Gal}_K^{\mathrm{ab}}$ with $\Theta_\ell$.
\end{definition}

\begin{proposition}\label{324} Let $\chi_\ell$ be a character of $\rho_\ell$ as above. If $\ell$ is sufficiently large, then there is a character $\Theta$ of $\mathbf{S}_{\mathfrak{m}_0}$ such that 
\begin{equation*}
\chi_\ell(x)= \theta_\ell(x)
\end{equation*}
for all $x\in U_{\mathfrak{m}_0}$, where $\theta_\ell$ is defined in Definition \ref{theta}.
\end{proposition}

\begin{proof} Since $0\leq m(\sigma,\ell)\leq c_6$ for all $\sigma\in\mathrm{Hom}(K,\bar{\Q})$ and $\ell\gg1$ by Proposition \ref{323}, the proposition follows by the proof of \cite[Proposition 20]{S72}. \end{proof}

Let $\Psi:\mathbf{S}_{\mathfrak{m}_0}\rightarrow \mathrm{GL}_{n,\mathbb{Q}}$ be a $\mathbb{Q}$-morphism of the Serre group $\mathbf{S}_{\mathfrak{m}_0}$ with finite kernel. Then $\Psi$ induces a strictly compatible system $\{\Psi_\ell\}_{\ell\in\mathscr{P}}$ of abelian $\ell$-adic representations of $\mathrm{Gal}_K$ \cite[Chapter 2]{Sbk} with $S=\emptyset$ (Definition \ref{comsys}):
\begin{equation*}
\Psi_\ell:\mathrm{Gal}_K\rightarrow \mathrm{Gal}_K^{\mathrm{ab}}\rightarrow \mathrm{GL}_n(\Q_\ell).
\end{equation*}
We may assume $\{\Psi_\ell\}$ is integral \cite[Chapter 2 $\mathsection3.4$]{Sbk} by twisting $\{\Psi_\ell\}$ with suitable big power of the system of cyclotomic characters.  

\begin{proposition}\label{325} Given $\Psi$ and $\{\Psi_\ell\}_{\ell\in\mathscr{P}}$ as above.
\begin{enumerate}
\item[(i)] The subgroup generated by the characters of $\Psi$ is of finite index in the character group of $\mathbf{S}_{\mathfrak{m}_0}$. Denote this index by $k$.
\item[(ii)] For any $\ell$ and character $\theta_\ell$ of $I_K$ induced from a character $\Theta$ of $\mathbf{S}_{\mathfrak{m}_0}$ in Definition \ref{theta},
 we obtain the following congruence for all $x\in U_{\mathfrak{m}_0}\subset I_K$
\begin{equation*}
\theta_\ell(x)\equiv \prod_{\sigma\in\mathrm{Hom}(K,\bar{\mathbb{Q}})}\sigma_\ell(x_\ell^{-1})^{m(\sigma)}\hspace{.3in}(\mathrm{mod}\hspace{.05in}\mathfrak{p}_\ell).
\end{equation*}
such that $ m(\sigma)\geq0$ for all $\sigma$.\end{enumerate}\end{proposition}

\begin{proof} Part (i) follows by $\Psi$ is an isogeny from $\mathbf{S}_{\mathfrak{m}_0}$ onto $\Psi(\mathbf{S}_{\mathfrak{m}_0})$. Part (ii) follows by the integrality of the system $\{\Psi_\ell\}$ and the theory of abelian $\ell$-adic representations \cite[Chapter 2,3]{Sbk}.\end{proof}
\vspace{.1in}

Denote the semi-simplification of some mod $\ell$ reduction of $\Psi_\ell$ by $\psi_\ell$ for all $\ell$. Consider the following strictly compatible system of $\ell$-adic representations
\begin{equation*}
\{\Phi_\ell\times\Psi_\ell:\mathrm{Gal}_K\rightarrow \mathrm{GL}_N(\Q_\ell)\times\mathrm{GL}_n(\Q_\ell)\}_{\ell\in\mathscr{P}}.
\end{equation*}
  The semi-simplification of some mod $\ell$ reduction of $\{\Phi_\ell\times\Psi_\ell\}_{\ell\in\mathscr{P}}$:
\begin{equation*}
\{\phi_\ell\times \psi_\ell: \mathrm{Gal}_K\rightarrow \mathrm{GL}_N(\mathbb{F}_\ell)\times\mathrm{GL}_n(\mathbb{F}_\ell)\}_{\ell\in\mathscr{P}}
\end{equation*}
is then a strictly compatible system of mod $\ell$ representations (Definition \ref{comsys}). Denote the image of $\phi_\ell\times\psi_\ell$ by $\bar{\Gamma}_\ell'$. Let $\bar{v}\in\Sigma_{\bar{K}}$ divide $\ell$. When we restrict $\phi_\ell\times\psi_\ell$ to inertia subgroup $I_{\bar{v}}$ of $\mathrm{Gal}_K$ and then semi-simplify, the exponents of characters of tame inertia quotient $I_{\bar{v}}^{\mathrm{t}}$ for some level are bounded independent of $\ell$ by $\mathsection2.3$, Proposition \ref{325}(ii), and \cite[Proposition 3]{S72}. Therefore, we can construct as in $\mathsection 2$  semisimple envelopes  $\{\bar{\mathbf{S}}_\ell'\}_{\ell\gg1}$ (Definition \ref{Nori}), inertia tori $\{\bar{\mathbf{I}}_\ell'\}_{\ell\gg1}$ (Theorem \ref{242}), and algebraic envelopes $\{\bar{\mathbf{G}}_\ell'\}_{\ell\gg1}$ (Definition \ref{env}) of $\{\bar{\Gamma}_\ell'\}_{\ell\gg1}$.

Since $\psi_\ell$ is semisimple and abelian, we see that Nori's construction gives $\bar{\mathbf{S}}_\ell'=\bar{\mathbf{S}}_\ell\times\{1\}\subset\mathrm{GL}_{N,\mathbb{F}_\ell}\times\mathrm{GL}_{n,\mathbb{F}_\ell}$. The normalizer of $\bar{\mathbf{S}}_\ell\times\{1\}$ in $\mathrm{GL}_{N,\mathbb{F}_\ell}\times\mathrm{GL}_{n,\mathbb{F}_\ell}$ is $\bar{\mathbf{N}}_\ell\times\mathrm{GL}_{n,\mathbb{F}_\ell}$. We have
\begin{equation*}
t_\ell\times\mathrm{id}:\bar{\mathbf{N}}_\ell\times\mathrm{GL}_{n,\mathbb{F}_\ell}\rightarrow \mathrm{GL}_{{W}_\ell}\times\mathrm{GL}_{n,\mathbb{F}_\ell}
\end{equation*}
with kernel $\bar{\mathbf{S}}_\ell\times\{1\}$. Therefore, we obtain a map
\begin{equation*}
\mu_\ell\times \psi_\ell:\mathrm{Gal}_K^\mathrm{ab}\rightarrow \mathrm{GL}({W}_\ell)\times\mathrm{GL}_n(\mathbb{F}_\ell)
\end{equation*}
with image denoted by $\bar{\Omega}_\ell'$. As $\bar{\Omega}_\ell'$ is abelian, denote the composition of $\mu_\ell$ and $\psi_\ell$ with $I_K\rightarrow \mathrm{Gal}_K^{\mathrm{ab}}$ by $\widetilde{\mu}_\ell$ and $\widetilde{\psi}_\ell$ for all $\ell$. By $(\ast)$ in the proof of Theorem \ref{242} and \cite[Proposition 9.5]{Neu}, we assume by taking a finite extension of $K$ that
\begin{equation*}
(\ast\ast):~~~(\widetilde{\mu}_\ell\times\widetilde{\psi}_\ell)(\prod_{v|\ell}U_v)=\bar{\Omega}_\ell'\hspace{.2in}\forall \ell\gg1.
\end{equation*}

\begin{proposition}\label{326} Let $p_2:\mathrm{GL}_{W_\ell}\times\mathrm{GL}_{n,\mathbb{F}_\ell}$ be the projection to the second factor. Then $p_2$ is an isogeny from $\bar{\mathbf{I}}_\ell'$ onto $p_2(\bar{\mathbf{I}}_\ell')$ for $\ell\gg1$. \end{proposition}

\begin{proof} Let $(x,1)\in \mathrm{GL}_{W_\ell}\times\mathrm{GL}_{n,\mathbb{F}_\ell}$ be an element of $ \bar{\Omega}_\ell'\cap \mathrm{Ker}(p_2)$, where $(x,1)=(\widetilde{\mu}_\ell\times\widetilde{\psi}_\ell)(x_\ell)$ for some $x_\ell\in\prod_{v|\ell}U_v$ (Definition \ref{Serre}) by $(\ast\ast)$ above. Since $\Psi:\mathbf{S}_{\mathfrak{m}_0}\rightarrow \mathrm{GL}_{n,\mathbb{Q}}$ has finite kernel and $\widetilde{\mu}_\ell\times\widetilde{\psi}_\ell$ is abelian and semisimple, we have $x^k=1$ for $\ell\gg1$ by $1=\widetilde{\psi}_\ell(x_\ell)$, 
Proposition \ref{324}, and Proposition \ref{325}(i). Since $\bar{\Omega}_\ell'$ is abelian of order prime to $\ell$, $x^k=1$ implies $x$ has at most $k^{\mathrm{dim}(W_\ell)}$ possibilities (by diagonalizing the image of $\widetilde{\mu}_\ell$) which implies 
\begin{equation*}
|\bar{\Omega}_\ell'\cap \mathrm{Ker}(p_2)|\leq k^{\mathrm{dim}(W_\ell)}.
\end{equation*} 
Therefore, the $\mathbb{F}_\ell$-diagonalizable group $\mathrm{Ker}(p_2)\cap \bar{\mathbf{I}}_\ell'$ cannot have positive dimension for $\ell\gg1$ because $[\bar{\mathbf{I}}_\ell'(\mathbb{F}_\ell):\bar{\Omega}_\ell'\cap \bar{\mathbf{I}}_\ell'(\mathbb{F}_\ell)]$ is also uniformly bounded by Theorem \ref{242}(ii). Thus, $p_2$ is an isogeny from $\bar{\mathbf{I}}_\ell'$ onto $p_2(\bar{\mathbf{I}}_\ell')$. \end{proof}

\noindent \textbf{\textit{Proof of Theorem \ref{321}.}}

\begin{proof} The mod $\ell$ system 
$$\{\phi_\ell\times\psi_\ell:\mathrm{Gal}_K\rightarrow \mathrm{GL}_N(\mathbb{F}_\ell)\times\GL_n(\F_\ell)\}$$ 
comes from the $\ell$-adic system (i.e., the semi-simplification of a mod $\ell$ reduction) 
$$\{ \Phi_\ell^{\ss}\times\Psi_\ell:\mathrm{Gal}_K\rightarrow \mathrm{GL}_N(\Q_\ell)\times\GL_n(\Q_\ell)\}.$$ 
Let $\mathbf{G}_\ell'$ be the algebraic monodromy group of semisimple representation $\Phi_\ell^{\ss}\times\Psi_\ell$ for all $\ell$. 
Thus, $\mathbf{G}_\ell'$ is reductive and we may assume $\mathbf{G}_\ell'$ is connected for all $\ell$ by taking a finite extension of $K$.
Denote the projection to the first and second factor of 
$\mathrm{GL}_{N}\times\mathrm{GL}_{n}$ by respectively $p_1$ and $p_2$. 
Consider the map
\begin{equation*}
\textit{Char}_1\times\textit{Char}_2:\GL_N\times\GL_n\rightarrow (\mathbb{G}_a^{N-1}\times\mathbb{G}_m) \times (\mathbb{G}_a^{n-1}\times\mathbb{G}_m)
\end{equation*}
where $\textit{Char}_i=\textit{Char}\circ p_i$, $i=1,2$. Note that the restriction of $\textit{Char}_1\times\textit{Char}_2$ to $\mathbb{G}_m^N\times\mathbb{G}_m^n$ is a finite morphism. Let $\mathbf{T}_\ell'$ be a maximal torus of monodromy group $\mathbf{G}_\ell'$ and 
$\bar{\mathbf{T}}_\ell'$ a maximal torus of $\bar{\mathbf{G}}_\ell'$, the algebraic envelope of the mod $\ell$ representation $\phi_\ell\times\psi_\ell$.
Up to conjugation by $\GL_N\times\GL_n$ (over algebraically closed fields), 
we may assume
$\mathbf{T}_\ell'$ 
and $\bar{\mathbf{T}}_\ell'$ are diagonal (i.e., inside $\mathbb{G}_m^{N+n}$).
We claim that up to permutation of coordinates
by $\mathrm{Perm}(N)\times\mathrm{Perm}(n)$,
$\mathbf{T}_\ell'$ 
and $\bar{\mathbf{T}}_\ell'$ are annihilated by the 
same set of characters of $\mathbb{G}_m^{N+n}$ 
for all sufficiently large $\ell$. 
The proof of the claim goes exactly the same as the proof of Theorem \ref{311}(i) with the following 
replacements: 
\begin{itemize}
\item[\textbullet] $\GL_N$ $\longrightarrow$ $\GL_N\times\GL_n$
\item[\textbullet] Morphism $\textit{Char}$ $\longrightarrow$ morphism $\textit{Char}_1\times\textit{Char}_2$
\item[\textbullet] $\Q$-variety $\textit{Char}(\mathbf{G}_\ell)$ $\longrightarrow$ $\Q$-variety $\textit{Char}_1\times\textit{Char}_2(\mathbf{G}_\ell')$
\item[\textbullet] $\mathrm{Perm}(N)$ $\longrightarrow$ $\mathrm{Perm}(N)\times\mathrm{Perm}(n)$
\end{itemize}

Therefore, $\mathbf{T}_\ell'':=\mathrm{Ker}(p_2:\mathbf{T}_\ell'\to p_2(\mathbf{T}_\ell'))^\circ$ and $\bar{\mathbf{T}}_\ell'':=\mathrm{Ker}(p_2:\bar{\mathbf{T}}_\ell'\to p_2(\bar{\mathbf{T}}_\ell'))^\circ$ as subtori of $\mathbb{G}_m^N$ are annihilated by the same set of characters for $\ell\gg1$.
Torus $\mathbf{T}_\ell''$ is the formal character of 
$\mathbf{G}_\ell^\mathrm{der}\hookrightarrow\mathrm{GL}_{N,\Q_\ell}$ \cite[proof of Theorem 3.19]{Hui}.
It suffices to show $\bar{\mathbf{T}}_\ell''$ is a maximal torus of $\bar{\mathbf{S}}_\ell$ for $\ell\gg1$.
Since the dimension of torus $\bar{\mathbf{I}}_\ell'$ is equal to the dimension of the center of algebraic envelope $\bar{\mathbf{G}}_\ell'$ for $\ell\gg1$ (see $\mathsection2.5$) and $p_2$ is an isogeny from $\bar{\mathbf{I}}_\ell'$ onto $p_2(\bar{\mathbf{I}}_\ell')$ by Proposition \ref{326} for $\ell\gg1$, the identity component of the kernel of 
$$p_2:\bar{\mathbf{G}}_\ell'\to p_2(\bar{\mathbf{G}}_\ell')=p_2(\bar{\mathbf{I}}_\ell')$$
is $\bar{\mathbf{S}}_\ell'$ (the semisimple part of $\bar{\mathbf{G}}_\ell'$) for $\ell\gg1$.  
Since
$p_2(\bar{\mathbf{T}}_\ell')=p_2(\bar{\mathbf{G}}_\ell')=p_2(\bar{\mathbf{I}}_\ell')$ for $\ell\gg1$,
$\bar{\mathbf{T}}_\ell''$ by construction is a maximal torus of $\bar{\mathbf{S}}_\ell'=\bar{\mathbf{S}}_\ell\times\{1\}$ for $\ell\gg1$.
Hence, the formal character of $\bar{\mathbf{S}}_\ell\hookrightarrow\GL_{N,\F_\ell}$ and $\mathbf{G}_\ell^{\der}\hookrightarrow\GL_{N,\Q_\ell}$ are the same for $\ell\gg1$. This proves (i). Since the formal character of 
$\mathbf{G}_\ell^{\der}\hookrightarrow\GL_{N,\Q_\ell}$ is independent of $\ell$ \cite[Theorem 3.19]{Hui}, we obtain (ii) by (i).
\end{proof}

\subsection{Proofs of Theorem \ref{main} and Corollary \ref{cor}}
The following purely representation theoretic result is crucial to the study of Galois images $\bar\Gamma_\ell$ for $\ell\gg1$.

\begin{theorem}\label{331}\cite[Theorem 2.19]{Hui}
Let $V$ be a finite dimensional $\C$-vector space and $\rho_1:\mathfrak{g}\to\End(V)$ and $\rho_2:\mathfrak{h}\to\End(V)$ are two faithful representations of complex semisimple Lie algebras. If the formal characters of $\rho_1$ and $\rho_2$ are equal, then the number of $A_n$ factors for $n\in\N\backslash\{1,2,3,4,5,7,8\}$ and the parity of $A_4$ factors of $\mathfrak{g}$ and $\mathfrak{h}$ are equal.
\end{theorem}

\begin{theorem}\label{332} The number of $A_n=\mathfrak{sl}_{n+1}$ factors for $n\in\N\backslash\{1,2,3,4,5,7,8\}$ and the parity of $A_4$ factors of $\bar{\mathbf{S}}_\ell\times_{\F_\ell}\bar{\F}_\ell$ are independent of $\ell$ if $\ell\gg1$.\end{theorem}

\begin{proof}
Let $\bar{\mathbf{S}}_\ell^{\sc}\rightarrow \bar{\mathbf{S}}_\ell$ be the simply connected cover of the semisimple $\bar{\mathbf{S}}_\ell$ for $\ell\gg1$. 
Then the representation $(\bar{\mathbf{S}}_\ell^{\sc}\rightarrow \bar{\mathbf{S}}_\ell\hookrightarrow \mathrm{GL}_{N,\mathbb{F}_\ell})\times\bar\F_\ell$ can be lifted to a representation of a simply connected Chevalley scheme $\mathbf{H}_{\ell,\Z}$ defined over $\mathbb{Z}$ for $\ell\gg1$ \cite[Theorem 24]{EHK}
\begin{equation*}
\pi_{\ell,\Z}: \mathbf{H}_{\ell,\Z} \rightarrow \GL_{N,\Z}
\end{equation*}
which is also a $\Z$-form of a representation of simply connected $\C$-semisimple group $\mathbf{H}_{\ell,\C}$ \cite{Steinbk}
\begin{equation*}
\pi_{\ell,\C}: \mathbf{H}_{\ell,\C} \rightarrow \GL_{N,\C}.
\end{equation*}
 Hence, $\bar{\mathbf{S}}_\ell\subset \mathrm{GL}_{N,\mathbb{F}_\ell}$ and $\pi_{\ell,\C}(\mathbf{H}_{\ell, \C})\subset \GL_{N,\C}$ have the same formal character for $\ell\gg1$. This and Theorem \ref{321} imply the formal character of $\pi_{\ell,\C}(\mathbf{H}_{\ell, \C})\subset \GL_{N,\C}$ is independent of $\ell$ when $\ell$ is sufficiently large. This in turn implies the formal character of $\mathrm{Lie}(\pi_{\ell,\C}(\mathbf{H}_{\ell, \C}))\hookrightarrow\End(\C^N)$ (see \cite[$\mathsection2.1$]{Hui}) is independent of $\ell$ when $\ell$ is sufficiently large. Therefore, the number of $A_n$ factors for $n\in\N\backslash\{1,2,3,4,5,7,8\}$ and the parity of $A_4$ factors of $\pi_{\ell,\C}(\mathbf{H}_{\ell, \C})$ and hence $\mathbf{H}_{\ell, \C}$ (the homomorphism $\mathbf{H}_{\ell, \C}\rightarrow \pi_{\ell,\C}(\mathbf{H}_{\ell, \C})$ is an isogeny since $\bar{\mathbf{S}}_\ell^{\sc}\rightarrow \bar{\mathbf{S}}_\ell$ is an isogeny) are independent of $\ell$ for $\ell\gg1$ by Theorem \ref{331}. Since the number of simple factors of each type of $\bar{\mathbf{S}}_\ell^{\sc}\times\bar{\F}_\ell$ and $\mathbf{H}_{\ell,\mathbb{C}}$ are equal, we are done.\end{proof}
\vspace{.1in}

Let $\mathfrak g$ be a simple Lie type (e.g., $A_n,B_n,C_n,D_n,...$) and $\bar\Gamma$ a finite group. Suppose $\ell\geq 5$. We measure the number of $\mathfrak g$-type simple factors of characteristic $\ell$ and the total number of Lie type simple factors of characteristic $\ell$ in the set of composition factors of $\bar\Gamma$ in the following sense: 
Let $\mathbb{F}_q$ be a finite field of characteristic $\ell$, $\sigma$ the Frobenius automorphism of $\bar{\mathbb{F}}_q/\mathbb{F}_q$, and $\bar{\mathbf{G}}$ a connected $\mathbb{F}_q$-group which is almost simple over $\bar\F_q$. 
The identification of $\bar{\mathbf{G}}_\sigma:=\bar{\mathbf{G}}(\mathbb{F}_q)$ is related to $\mathfrak{g}$, the simple type of $\bar{\mathbf{G}}\times_{\F_q}\bar\F_q$ \cite[11.6]{Stein}:

\begin{center}
\begin{tabular}{c|c} 
Type of $\bar{\mathbf{G}}$  & Composition factors of $\bar{\mathbf{G}}(\mathbb{F}_q)$ \\ \hline

$A_1$  & $A_1(q)=\PSL_2(q)$  + cyclic groups \\ 

$A_n$ ($n\geq 2$) & $A_n(q)$ or ${}^2\!A_n(q^{2})$ + cyclic groups \\ 

$B_n$ ($n\geq 2$) & $B_n(q)$ + cyclic groups \\

$C_n$ ($n\geq 3$) & $C_n(q)$  + cyclic groups \\ 

$D_4$  & $D_4(q)$ or ${}^2\!D_4(q^{2})$ or ${}^3\!D_4(q^{3})$ + cyclic groups \\ 

$D_n$ ($n\geq 5$) & $D_n(q)$ or ${}^2\!D_n(q^{2})$ + cyclic groups \\ 

$E_6$  & $E_6(q)$ or ${}^2\!E_6(q^{2})$ + cyclic groups \\ 

$E_7$  & $E_7(q)$  + cyclic groups \\ 

$E_8$  & $E_8(q)$  + cyclic groups \\ 

$F_4$  & $F_4(q)$  + cyclic groups \\ 

$G_2$  & $G_2(q)$  + cyclic groups \\ 
\end{tabular}
\end{center}

$\bar{\mathbf{G}}(\mathbb{F}_q)$ has only one non-cyclic composition factor which is either a Chevalley group or a Steinberg group of  type $\mathfrak{g}$. For example, the non-cyclic composition factor is $A_n(q)$ or ${}^2\!A_n(q^{2})$ if $\mathfrak{g}=A_n$ and $n\geq2$. For any semisimple algebraic group $\mathbf{H}/F$ and complex semisimple Lie algebra $\mathfrak{h}$, denote by $\rk\mathbf{H}$ and $\rk\mathfrak{h}$ respectively the rank of $\mathbf{H}/\bar F$ and the rank of $\mathfrak{h}$.

\begin{definition}\label{rank} Suppose $\ell\geq 5$ is a prime number and $q=\ell^f$. Let $\bar\Gamma$ be a finite simple group of Lie type (of characteristic $\ell$) in the above table and $\mathfrak g$ the simple Lie type of the corresponding $\bar{\mathbf{G}}$. We define \emph{the $\mathfrak g$-type $\ell$-rank} of $\bar\Gamma$ to be 
\begin{equation*}
\mathrm{rk}_\ell^{\mathfrak g}\bar\Gamma:= \left\{ \begin{array}{lll}
 f\cdot\rk\mathfrak{g}& \mathrm{if} ~\bar\Gamma~ \mathrm{is~associated~with} ~\mathfrak{g} ~\mathrm{in~the~above~table},\\
 0 & \mathrm{otherwise.}
\end{array}\right.
\end{equation*}
For finite simple group $\bar\Gamma'$ not in the table, $\mathrm{rk}_\ell^{\mathfrak g}\bar\Gamma'$ is defined to be $0$ for any $\mathfrak g$. We extend this definition
to arbitrary finite groups by defining the $\mathfrak g$-type $\ell$-rank of any finite group to be the sum of the $\mathfrak g$-type $\ell$-ranks of its composition factors. \emph{The total $\ell$-rank} of a finite group $\bar\Gamma$ is defined to be 
\begin{equation*}
\mathrm{rk}_\ell\bar\Gamma:=\sum_{\mathfrak g}\mathrm{rk}_\ell^{\mathfrak g}\bar\Gamma.
\end{equation*}
\end{definition}

\begin{remark}\label{333}
The definition of $\mathfrak{g}$-type $\ell$-rank is equivalent to the following.
For any  finite simple group $\bar\Gamma$ of Lie type of characteristic $\ell$, we have 
$$\bar\Gamma=\bar{\mathbf{G}}(\F_{\ell^{f'}})^{\der}$$
for some adjoint simple group $\bar{\mathbf{G}}/\F_{\ell^{f'}}$
so that 
$$\bar{\mathbf{G}}\times_{\F_{\ell^{f'}}}\bar\F_\ell=\prod^m\bar{\mathbf{H}},$$
where $\bar{\mathbf{H}}$
is an $\bar\F_\ell$-adjoint simple group of some Lie type $\mathfrak{h}$.
We then set the $\mathfrak{g}$-type $\ell$-rank of $\bar\Gamma$ to be
\begin{equation*}
\rk^{\mathfrak{g}}_\ell\bar\Gamma :=\left\{ \begin{array}{lll}
 f'\cdot\rk \bar{\mathbf{G}} &\mbox{if}\hspace{.1in} \mathfrak{g}=\mathfrak{h}.\\
 0 &\mbox{otherwise.}
\end{array}\right.
\end{equation*}
We extend this definition
to arbitrary finite groups by defining the $\mathfrak g$-type $\ell$-rank of any finite group to be the sum of the $\mathfrak g$-type $\ell$-ranks of its composition factors.\\
\end{remark}

Let $\bar{\mathbf{G}}$ be a connected semisimple algebraic group over $\mathbb{F}_q$ and $\pi:\bar{\mathbf{G}}^{\sc}\rightarrow \bar{\mathbf{G}}$ the simply-connected cover of $\bar{\mathbf{G}}$. Simply-connected $\bar{\mathbf{G}}^{\sc}$ and isogeny $\pi$ are defined over $\mathbb{F}_q$ \cite[9.16]{Stein}. Group $\bar{\mathbf{G}}^{\sc}$ is a direct product of $\F_q$-simple, simply-connected semisimple groups $\bar{\mathbf{G}}_i^{\sc}$  \cite[Chapter 10 $\mathsection 1.3$]{CF}:
\begin{equation*}
\bar{\mathbf{G}}_1^{\sc} \times \bar{\mathbf{G}}_2^{\sc} \times\cdots\times \bar{\mathbf{G}}_k^{\sc} \stackrel{\mathbb{F}_q\cong}{\longrightarrow} \bar{\mathbf{G}}^{\sc}.
\end{equation*}
For each $\bar{\mathbf{G}}_i^{\sc}$, there exist an integer $m_i$ and an algebraic group $\bar{\mathbf{H}}_i^{\sc}$ defined over $\F_{q^{m_i}}$ such that $\bar{\mathbf{H}}_i^{\sc}\times_{\F_{q^{m_i}}}\bar\F_q$ is almost simple and
$$\bar{\mathbf{G}}_i^{\sc}\times_{\F_q}\F_{q^{m_i}}=\prod^{m_i}\bar{\mathbf{H}}_i^{\sc}.$$
We have \cite[Chapter 10 $\mathsection 1.3$]{CF}
\begin{equation*}
\bar{\mathbf{G}}_i^{\sc} = \mathrm{Res}_{\mathbb{F}_{q^{m_i}}/\mathbb{F}_{q}}(\bar{\mathbf{H}}_i^{\sc})
\end{equation*}
so that
\begin{equation*}
\bar{\mathbf{G}}_i^{\sc}(\mathbb{F}_q) = \bar{\mathbf{H}}_i^{\sc}(\mathbb{F}_{q^{m_i}}).
\end{equation*}
The following proposition relates $\rk_\ell^\mathfrak{g}\bar{\mathbf{G}}(\F_q)$ and $\rk_\ell\bar{\mathbf{G}}(\F_q)$ to  $\bar{\mathbf{G}}\times_{\F_q}\bar{\F}_q$.

\begin{proposition}\label{334} Let $\ell\geq5$ be a prime and $\bar{\mathbf{G}}$ a connected semisimple algebraic group over $\mathbb{F}_q$, where $q=\ell^f$. The composition factors of $\bar{\mathbf{G}}(\mathbb{F}_q)$ are cyclic groups and finite simple groups of Lie type of characteristic $\ell$. Moreover, let $m$ be the number of almost simple factors of $\bar{\mathbf{G}}\times_{\F_q}\bar{\F}_q$ of simple type $\mathfrak{g}$.
Then, 
$$\rk_\ell^\mathfrak{g}\bar{\mathbf{G}}(\F_q)=mf\cdot\rk\mathfrak{g}\hspace{.2in}\mathrm{and}\hspace{.2in}\rk_\ell\bar{\mathbf{G}}(\F_q)=f\cdot\rk\bar{\mathbf{G}}.$$
\end{proposition}

\begin{proof} Since the kernel and the cokernel of $\pi:\bar{\mathbf{G}}^{\sc}(\mathbb{F}_q)\rightarrow \bar{\mathbf{G}}(\mathbb{F}_q)$ are both abelian \cite[12.6]{Stein},  the composition factors of $\bar{\mathbf{G}}(\mathbb{F}_q)$ and $\prod_{i=1}^k\bar{\mathbf{H}}_i^{\sc}(\mathbb{F}_{q^{m_i}})$ defined above are identical modulo cyclic groups.
Hence,
the composition factors of $\bar{\mathbf{G}}(\mathbb{F}_q)$ are cyclic groups and finite simple groups of Lie type of characteristic $\ell$ by the table.
Let \begin{equation*}
\{\bar{\mathbf{H}}_1^{\sc},\bar{\mathbf{H}}_2^{\sc},...,\bar{\mathbf{H}}_j^{\sc}\}
\end{equation*}
be the subset of $\{\bar{\mathbf{H}}_1^{\sc},...,\bar{\mathbf{H}}_k^{\sc}\}$ of type $\mathfrak{g}$. The equation 
$$m_1+m_2+\cdots+ m_j=m$$
follows immediately from the fact that each $\bar{\mathbf{G}}_i^{\sc}$ is a direct product of $m_i$ copies of $\bar{\mathbf{H}}_i^{\sc}$ over $\bar\F_q$. Since $\bar{\mathbf{H}}_i^{\sc}$ is almost simple over $\bar\F_q$, we obtain by Definition \ref{rank} that
the $\mathfrak{g}$-type $\ell$-rank
$$\rk_\ell^\mathfrak{g}\bar{\mathbf{G}}(\F_q)= \sum_{i=1}^k \rk_\ell^\mathfrak{g}\bar{\mathbf{H}}_i^{\sc}(\mathbb{F}_{q^{m_i}}) =\sum_{i=1}^j  m_if\cdot\rk\mathfrak{g}=mf\cdot\rk\mathfrak{g}.$$
and therefore the total $\ell$-rank 
$$\rk_\ell\bar{\mathbf{G}}(\F_q)=f\cdot\rk\bar{\mathbf{G}}.$$\end{proof}

We can now prove our main results.

\begin{customthm}{A}(Main Theorem)  \textit{Let $K$ be a number field and $\{\phi_\ell:\mathrm{Gal}_K\rightarrow\mathrm{GL}_N(\mathbb{F}_\ell)\}_{\ell\in\mathscr{P}}$ a strictly
compatible system of mod $\ell$ Galois representations  arising from \'etale cohomology (Definition \ref{arise},\ref{comsys}). There exists a finite normal extension $L$ of $K$ such that if we denote $\phi_\ell(\mathrm{Gal}_K)$ and $\phi_\ell(\mathrm{Gal}_L)$ by respectively $\bar\Gamma_\ell$ and $\bar{\gamma}_\ell$ for all $\ell$, and let $\bar{\mathbf{S}}_\ell\subset\GL_{N,\F_\ell}$ be the connected $\mathbb{F}_\ell$-semisimple subgroup associated to $\bar{\gamma}_\ell$ (or $\bar\Gamma_\ell$) by Nori's theory for $\ell\gg1$, then the following hold for $\ell\gg1:$
\begin{enumerate}
\item[(i)] The formal character of $\bar{\mathbf{S}}_\ell\hookrightarrow \mathrm{GL}_{N,\mathbb{F}_\ell}$ is independent of $\ell$ (Definition \ref{fc}') and is equal to the formal character of  $(\mathbf{G}_\ell^\circ)^{\mathrm{der}}\hookrightarrow \mathrm{GL}_{N,\Q_\ell}$, where $(\mathbf{G}_\ell^\circ)^{\mathrm{der}}$ is the derived group of the identity component of $\mathbf{G}_\ell$, the algebraic monodromy group of the semi-simplified representation $\Phi_\ell^{\ss}$.
\item[(ii)] The composition factors of $\bar{\gamma}_\ell$ and $\bar{\mathbf{S}}_\ell(\mathbb{F}_\ell)$ are identical modulo cyclic groups. Therefore, the composition factors of $\bar{\gamma}_\ell$ are finite simple groups of Lie type of characteristic $\ell$ and cyclic groups. 
\end{enumerate}}
\end{customthm}

\begin{proof} By Proposition \ref{N2}(i), $\bar{\mathbf{S}}_\ell\subset\GL_{N,\F_\ell}$ 
is a connected $\mathbb{F}_\ell$-semisimple subgroup for $\ell\gg1$. 
Part (i) is proved by Theorem \ref{321}. Since there is a finite normal extension $L/K$  
such that $\bar{\gamma}_\ell:=\phi_\ell(\mathrm{Gal}_L)$ is a subgroup of $\bar{\mathbf{G}}_\ell(\mathbb{F}_\ell)$ 
of uniform bounded index by Theorem \ref{redenv} and $\bar{\mathbf{S}}_\ell$ 
is the derived group of $\bar{\mathbf{G}}_\ell$, the composition factors of 
$\bar{\gamma}_\ell$ and $\bar{\gamma}_\ell\cap \bar{\mathbf{S}}_\ell(\mathbb{F}_\ell)$ 
are identical modulo cyclic groups. 
Together with $\bar{\mathbf{S}}_\ell(\mathbb{F}_\ell)/\bar{\mathbf{S}}_\ell(\mathbb{F}_\ell)^+$ abelian and normal series
\begin{equation*}
  \bar{\mathbf{S}}_\ell(\mathbb{F}_\ell)^+=\bar{\gamma}_\ell^+ \hspace{.05in} \triangleleft\hspace{.05in}  \bar{\gamma}_\ell\cap \bar{\mathbf{S}}_\ell(\mathbb{F}_\ell)\hspace{.05in} \triangleleft \hspace{.05in}  \bar{\mathbf{S}}_\ell(\mathbb{F}_\ell)
\end{equation*}
for $\ell\gg1$ by Theorem \ref{N1} and Remark \ref{ind}, we conclude that the composition factors of $\bar{\gamma}_\ell$ and $\bar{\mathbf{S}}_\ell(\F_\ell)$ are identical modulo cyclic groups. Since Proposition \ref{334} implies the non-cyclic composition factors of $\bar{\mathbf{S}}_\ell(\mathbb{F}_\ell)$ are finite simple groups of Lie type of characteristic $\ell$, we obtain (ii).\end{proof}

\begin{customcor}{B} \textit{Let $\bar{\mathbf{S}}_\ell$ be defined as above, then the following hold for $\ell\gg1:$
\begin{enumerate} 
\item[(i)] The total $\ell$-rank $\rk_\ell\bar{\Gamma}_\ell$ of $\bar\Gamma_\ell$ (Definition \ref{rank}) is equal to the rank of $\bar{\mathbf{S}}_\ell$ and is therefore independent of $\ell$.
\item[(ii)] The $A_n$-type $\ell$-rank $\rk_\ell^{A_n}\bar{\Gamma}_\ell$ of $\bar\Gamma_\ell$ (Definition \ref{rank}) for $n\in\N\backslash\{1,2,3,4,5,7,8\}$ and the parity of 
 $(\rk_\ell^{A_4}\bar{\Gamma}_\ell)/4$ are independent of $\ell$. 
\end{enumerate}}
\end{customcor}

\begin{proof} Since $\bar\gamma_\ell$ is a normal subgroup of $\bar\Gamma_\ell$ of index bounded by $[L:K]$,
they have equal total $\ell$-rank and $\mathfrak{g}$-type $\ell$-rank for all sufficiently large $\ell$. It suffices to prove
(i) and (ii) for $\bar\gamma_\ell$.
Since (i) is a direct consequence of Proposition \ref{334} and Theorem \ref{main} and (ii) follows easily from Theorem \ref{332}, Proposition \ref{334}, and Theorem \ref{main}, we are done.\end{proof}

\section*{Acknowledgments} It is my great pleasure to acknowledge my adviser, Michael Larsen, for useful conversations during the course of the work and helpful comments on an earlier preprint of this project. 
I am also grateful to the anonymous referee for many helpful comments and suggestions, which largely improved the exposition of the paper.

\end{document}